\documentclass[12pt,leqno,a4paper,final]{amsart}
\usepackage{amsmath, verbatim, color}
\usepackage{mathrsfs}
\usepackage{dsfont}%this is for nicer double-stroke letters
\usepackage[a4paper,margin=0.91in]{geometry}%this is for easy margins

\usepackage[notcite,notref]{showkeys}

\theoremstyle{plain}
\newtheorem{theorem}{Theorem}[section]
\newtheorem{corollary}[theorem]{Corollary}
\newtheorem{lemma}[theorem]{Lemma}
\newtheorem{proposition}[theorem]{Proposition}
\theoremstyle{definition}

\newtheorem{example}[theorem]{Example}

\numberwithin{equation}{section}

\newcommand\E{\mathds E}

\newcommand\N{\mathds N}
\newcommand\R{\mathds R}

\newcommand\I{\mathds 1}
\renewcommand\P{\mathds P}

\newcommand\Bscr{\mathscr{B}}

\newtheorem*{ack}{Acknowledgement}

\renewcommand\d{\mathrm{d}}
\newcommand\e{\mathrm{e}}

\begin{document}\allowdisplaybreaks
\title[Log-Harnack Inequalities for Gruschin Type
Semigroups]{\bfseries Log-Harnack
Inequalities for Markov Semigroups Generated by Non-Local
Gruschin Type Operators}

\author[C.-S.~Deng]{Chang-Song Deng}
\address[C.-S.~Deng]{School of Mathematics and Statistics\\ Wuhan University\\ Wuhan 430072, China}
\email{dengcs@whu.edu.cn}
\thanks{Financial support through the
National Natural Science Foundation
of China (11401442) (for Chang-Song Deng)
and the National Natural Science Foundation
of China (11626237) (for Shao-Qin Zhang)
is gratefully acknowledged.}

\author[S.-Q.~Zhang]{Shao-Qin Zhang}
\address[S.-Q.~Zhang]{School of Statistics and Mathematics, Central University of Finance and Economics, Beijing 100081, China}
\email{zhangsq@cufe.edu.cn}


\maketitle

\begin{abstract}
    Based on coupling in two steps and the regularization
    approximations of the underlying subordinators,
    we establish log-Harnack inequalities for Markov
    semigroups generated by a class of non-local
    Gruschin type operators. Some concrete examples
    are also presented.
\end{abstract}

\section{Introduction}

    The classical Gruschin semigroup on $\R^2$ with order $l>0$
    is generated by the differential operator
    $$
        L(x^{(1)},x^{(2)})=\frac12\left[
        \frac{\partial^2}{\partial(x^{(1)})^2}
        +|x^{(1)}|^{2l}
        \frac{\partial^2}{\partial(x^{(2)})^2}
        \right].
    $$
    The derivative formula of Bismut-Elworthy-Li's
    type (cf.\ \cite{Bis84, EL94})
    and log-Harnack
    inequality, first introduced
    in \cite{RW10}, have been investigated for
    the associated diffusion processes
    in \cite{Wan14} and \cite{WX14}, respectively.
    As a natural extension, let us consider the following
    non-local Gruschin type operator
    \begin{equation}\label{operator}
        L_{\phi_1,\phi_2}(x^{(1)},x^{(2)})
        =-\frac12\left[
        \phi_1\left(-
        \frac{\partial^2}{\partial(x^{(1)})^2}
        \right)
        +|x^{(1)}|^{2l}\phi_2
        \left(-\frac{\partial^2}{\partial(x^{(2)})^2}
        \right)
        \right].
    \end{equation}
    Here, each $\phi_i:(0,\infty)\rightarrow(0,\infty)$
    is a Bernstein function with
    $\lim_{u\downarrow0}\phi_i(u)=0$, i.e.\ $\phi_i$ is
    given by (see e.g.\ \cite[Theorem 3.2]{SSV12})
    \begin{equation}\label{bern}
        \phi_i(u)=\vartheta_i u+
        \int_{(0,\infty)}\left(1-\e^{-ux}\right)
        \,\nu_i(\d x),\quad u>0,
    \end{equation}
    where $\vartheta_i\geq0$ is the drift parameter, and
    $\nu_i$ is a L\'{e}vy measure, that is, a Radon measure on
    $(0,\infty)$ such that
    $\int_{(0,\infty)}(1\wedge x)\,\nu_i(\d x)
    <\infty$. The Markov process with jumps
    generated by the non-local
    operator \eqref{operator} can
    be constructed by solving
    the degenerate SDE driven by subordinate Brownian motions
    $$
        \left\{
        \begin{array}{l}
            \d X_t^{(1)}=\d W_{S_1(t)}^{(1)},\\
            \d X_t^{(2)}=|X_t^{(1)}|^l\,\d W_{S_2(t)}^{(2)},
        \end{array}
        \right.
    $$
    where $W_{t}^{(1)}$, $W_{t}^{(2)}$, $S_1(t)$, and
    $S_2(t)$ are independent processes such that
    each $W_{t}^{(i)}$ is a standard $1$-dimensional Brownian
    motion, and $S_i(t)$ is
    a subordinator (i.e.\ a non-decreasing L\'{e}vy process
    on $[0,\infty)$) determined by its Laplace
    transform which is of the form
    $$
        \E\,\e^{-uS_i(t)}=\e^{-t\phi_i(u)},
        \quad t\geq0,\,u>0.
    $$

    Due to the importance both in theory and in applications,
    recently, there has been considerable interest in the
    study of discontinuous Markov processes.
    The central aim of this paper is to
    establish the log-Harnack
    inequality for Markov semigroups
    generated by the non-local Gruschin type
    operator \eqref{operator}. The log-Harnack
    inequality can be regarded as a weaker version
    of F.-Y.\ Wang's dimension-free Harnack
    inequality with power initialed in \cite{Wan97},
    and has been thoroughly investigated, especially
    for diffusion processes; the basic
    argument was a coupling by change of
    measure, see \cite{Wbook} and reference
    therein for recent developments
    on Harnack type inequalities
    for various models.
    Since it is usually very
    difficult to construct successful couplings
    for non-linear SDEs driven by pure jump
    noises, the methods from diffusions cannot
    be directly applied and we need some technique from
    the study for jump-diffusion processes.
    In this article, our tool is based
    on the coupling approach in \cite{WX14}
    and the regularization
    approximations of time-changes used
    in \cite{Zha13, WW14, WZ15, Den14, DS16}.

    The log-Harnack inequality has become an efficient tool in stochastic analysis, and it can be used to study
    the strong Feller property, heat
    kernel estimates, transportation-cost
    inequalities, and many more; we refer to the
    monograph by F.-Y.\ Wang
    \cite[Subsection 1.4.1]{Wbook} for an
    in-depth explanation of its applications.

    For generality, we consider the following SDE
    for $X_t=(X_t^{(1)},X_t^{(2)})$ on
    $\R^{m+d}=\R^m\times\R^d$ ($m,d\in\N$):
    \begin{equation}\label{maineq}
        \left\{
        \begin{array}{l}
            \d X_t^{(1)}=\sigma_t\,\d W_{S_1(t)}^{(1)},\\
            \d X_t^{(2)}=b(t,X_t^{(2)})\,\d t
            +|X_t^{(1)}|^l\,\d W_{S_2(t)}^{(2)},
        \end{array}
        \right.
    \end{equation}
    where $\sigma:[0,\infty)\rightarrow\R^m\otimes\R^m$ is
    measurable and locally bounded, $b:[0,\infty)\times\R^d\rightarrow\R^d$ is
    measurable, locally bounded in the time
    variable $t\geq0$ and
    continuous in the
    space variable $x^{(2)}\in\R^d$,
    and $W_t:=(W_{t}^{(1)},W_{t}^{(2)})$, $S_1(t)$, and
    $S_2(t)$ are independent processes on
    a probability space $(\Omega,\mathscr{A}, \P)$
    such that
    \begin{enumerate}
        \item[(i)] $W_t$ is a
        standard Brownian motion
        on $\R^{m+d}$;

        \item[(ii)] Each $S_i(t)$ is a subordinator
            with characteristic
            exponent (Bernstein function) $\phi_i$
            given by \eqref{bern}.
    \end{enumerate}

    We will assume the following conditions on $\sigma$ and $b$:
    \begin{enumerate}
        \item[\textbf{(H1)}]
            For every $t\geq0$, $\sigma_t$ is
            invertible and there exists a non-decreasing
            function $\lambda: [0,\infty)\rightarrow(0,\infty)$
            such that $\|\sigma_t^{-1}\|\leq\lambda_t$
            for all $t\geq0$.

        \item[\textbf{(H2)}]
            There exists a locally bounded measurable function
            $k:[0,\infty)\rightarrow\R$ such that
                $$
                    \big\langle b(t,x^{(2)})
                    -b(t,y^{(2)}),x^{(2)}-y^{(2)}
                    \big\rangle\leq k(t)|x^{(2)}-y^{(2)}|^2,\quad x^{(2)},y^{(2)}\in\R^d,\,t\geq0.
                $$
    \end{enumerate}
    It is easy to see that once $X_t^{(1)}$ is fixed,
    then \textbf{(H2)} implies the existence, uniqueness and
    non-explosion of the solution to the second equation
    in \eqref{maineq}. For $x=(x^{(1)},x^{(2)})\in\R^m\times\R^d$,
    denote by $X_t(x)=(X_t^{(1)}(x),X_t^{(2)}(x))$ the solution
    to \eqref{maineq} with $X_0=x$. We aim to establish log-Harnack
    inequalities for the associated Markov semigroup
    $P_t$ on $\Bscr_b(\R^{m+d})$:
    $$
        P_tf(x):=\E f\big(X_t(x)\big),\quad
        t\geq0,\,f\in\Bscr_b(\R^{m+d}), \,x\in\R^{m+d}.
    $$

    In order to state our main result,
    we need the following notation:
    $$
        K(s,t):=\int_s^tk(r)\,\d r,\quad 0\leq s\leq t,
    $$
    where $k$ is the function appearing in \textbf{(H2)}.

\begin{theorem}\label{main1}
        Let $l\in(0,m/2)$ and
        assume that \textbf{\upshape(H1)} and
        \textbf{\upshape(H2)} hold. There is
        some constant $C=C(m,d,l)>0$ such
        that for any $T>0$, $x=(x^{(1)},x^{(2)}),
        y=(y^{(1)},y^{(2)})\in\R^{m+d}$, and
        $f\in\Bscr_b(\R^d)$ with $f\geq1$,
        \begin{align*}
            &P_{2T}\log f(y)
            \leq\log P_{2T}f(x)
            +\frac{|x^{(1)}-y^{(1)}|^2}{2}
            \,\E\left(
            \int_0^T\lambda_r^{-2}\,
            \d S_1(r)
            \right)^{-1}\\
            &\quad+C\e^{2K(0,T)}\Bigg\{
            |x^{(2)}-y^{(2)}|^2
            \E\left(
            \int_0^T\lambda_r^{-2}\,
            \d S_1(r)
            \right)^{-l}
            \cdot
            \E\left(
            \int_T^{2T}\e^{-2K(T,s)}
            \,\d S_2(s)
            \right)^{-1}\\
            &\qquad+\left(\left[
            |x^{(1)}|^{2(l-1)^+}+
            |y^{(1)}|^{2(l-1)^+}
            \right]
            \E \left(
            \int_0^T\lambda_r^{-2}\,
            \d S_1(r)
            \right)
            ^{-l}
            +
            \E \left(
            \int_0^T\lambda_r^{-2}\,
            \d S_1(r)
            \right)
            ^{-(l\wedge1)}
            \right)\\
            &\qquad\quad
            \times
            |x^{(1)}-y^{(1)}|^{2(l\wedge1)}
            \E\frac{\int_0^{T}\e^{-2K(0,s)}
            \,\d S_2(s)}
            {\int_T^{2T}\e^{-2K(T,s)}
            \,\d S_2(s)}
            \Bigg\}.
        \end{align*}
    \end{theorem}

The following result is a direct consequence of
Theorem \ref{main1}.

\begin{corollary}\label{cor}
        Let $\sigma= I_{m\times m}$, $b=0$,
        and $l\in(0,m/2)$. There is
        some constant $C=C(m,d,l)>0$ such
        that for any $T>0$, $x=(x^{(1)},x^{(2)}),
        y=(y^{(1)},y^{(2)})\in\R^{m+d}$, and
        $f\in\Bscr_b(\R^d)$ with $f\geq1$,
        \begin{align*}
            &P_{2T}\log f(y)
            \leq\log P_{2T}f(x)
            +\frac{|x^{(1)}-y^{(1)}|^2}
            {2}\E S_1(T)^{-1}\\
            &\;+C\E S_2(T)
            ^{-1}\cdot\bigg\{
            |x^{(2)}-y^{(2)}|^2
            \E S_1(T)^{-l}
            \\
            &\;\quad+\left(\left[
            |x^{(1)}|^{2(l-1)^+}+
            |y^{(1)}|^{2(l-1)^+}
            \right]
            \E S_1(T)
            ^{-l}
            +
            \E S_1(T)
            ^{-(l\wedge1)}
            \right)
            |x^{(1)}-y^{(1)}|^{2(l\wedge1)}
            \E S_2(T)
            \bigg\}.
        \end{align*}
    \end{corollary}

Now we apply our result to some concrete examples
of subordinators.

\begin{example}\label{ex1}
        Let $\sigma= I_{m\times m}$, $b=0$,
        and $l\in(0,m/2)$. Assume that
        $S_1$ is an $\alpha$-stable
        subordinator, which has no drift and its
        L\'{e}vy measure is given by $c_1x^{-1-\alpha}
        \I_{\{x>0\}}\d x$, and $S_2$ is
        a truncated $\beta$-stable
        subordinator, which has no drift and its L\'{e}vy
        measure is given by $c_2x^{-1-\beta}
        \I_{\{0<x<1\}}\d x$, where $\alpha,\beta\in(0,1)$
        and $c_1,c_2>0$
        are constants. Then there is
        some constant $C=C(m,d,l,\alpha,\beta,c_1,c_2)>0$ such
        that for any $T>0$, $x=(x^{(1)},x^{(2)}),
        y=(y^{(1)},y^{(2)})\in\R^{m+d}$, and
        $f\in\Bscr_b(\R^d)$ with $f\geq1$,
        \begin{align*}
            P_{2T}\log f(y)
            &\leq\log P_{2T}f(x)
            +C\bigg\{T^{-1/\alpha}|x^{(1)}-y^{(1)}|^2
            +T^{-1-l/\alpha}|x^{(2)}-y^{(2)}|^2\\
            &\qquad+\left(\left[
            |x^{(1)}|^{2(l-1)^+}+
            |y^{(1)}|^{2(l-1)^+}
            \right]T^{-l/\alpha}
            +T^{-(l\wedge1)/\alpha}\right)
            |x^{(1)}-y^{(1)}|^{2(l\wedge1)}
            \bigg\}.
        \end{align*}
    \end{example}

    \begin{example}\label{ex2}
        Let $\sigma= I_{m\times m}$, $b=0$,
        and $l\in(0,m/2)$. Assume that
        $S_1$ is an $\alpha$-stable
        subordinator, which has no drift and its
        L\'{e}vy measure is given by $c_1x^{-1-\alpha}
        \I_{\{x>0\}}\,\d x$,
        and $S_2$ is a relativistic
        $\beta$-stable subordinator, which has no
        drift and its L\'{e}vy
        measure is given by $c_2\e^{-\rho^\beta x}
        x^{-1-\beta}\I_{\{x>0\}}\,\d x$, where $\alpha,\beta\in(0,1)$
        and $c_1,c_2,\rho>0$
        are constants. Then there is
        some constant $C=C(m,d,l,\alpha,\beta,c_1,c_2,
        \rho)>0$ such
        that for any $T>0$, $x=(x^{(1)},x^{(2)}),
        y=(y^{(1)},y^{(2)})\in\R^{m+d}$, and
        $f\in\Bscr_b(\R^d)$ with $f\geq1$,
        \begin{align*}
            &P_{2T}\log f(y)
            \leq\log P_{2T}f(x)
            +C\bigg\{T^{-1/\alpha}|x^{(1)}-y^{(1)}|^2
            +T^{-l/\alpha}
            \left(T^{-1/\beta}\vee T^{-1}\right)
            |x^{(2)}-y^{(2)}|^2\\
            &\qquad\,+
            \left(T^{1-1/\beta}\vee1\right)
            \left(\left[
            |x^{(1)}|^{2(l-1)^+}+
            |y^{(1)}|^{2(l-1)^+}
            \right]T^{-l/\alpha}
            +T^{-(l\wedge1)/\alpha}\right)
            |x^{(1)}-y^{(1)}|^{2(l\wedge1)}
            \bigg\}.
        \end{align*}
    \end{example}

    The remaining part of this paper is organized as follows.
    By using coupling in two steps and an
    approximation argument, we establish in
    Section \ref{sec2} the log-Harnack inequalities
    for SDEs driven by non-random time-changed Brownian
    motions. Section \ref{sec3} is devoted to the proofs of
    Theorem \ref{main1} and Examples \ref{ex1}
    and \ref{ex2}.

\section{Log-Harnack inequalities under
deterministic time-changes}\label{sec2}

    For each $i\in\{1,2\}$,
    let $\ell_i:[0,\infty)\rightarrow[0,\infty)$ be
    a non-decreasing and c\`{a}dl\`{a}g function with
    $\ell_i(0)=0$. By \textbf{(H2)}, the following
    SDE for $X_t^{\ell_1,\ell_2}=(X_t^{(1),\ell_1},
    X_t^{(2),\ell_1,\ell_2})$
    has a unique non-explosive solution:
    \begin{equation}\label{changeeq}
        \left\{
        \begin{array}{l}
            \d X_t^{(1),\ell_1}=\sigma_t\,\d W_{\ell_1(t)}^{(1)},\\
            \d X_t^{(2),\ell_1,\ell_2}=b(t,X_t^{(2),\ell_1,\ell_2})
            \,\d t
            +|X_t^{(1),\ell_1}|^l\,\d W_{\ell_2(t)}^{(2)}.
        \end{array}
        \right.
    \end{equation}
    Let
    $$
        P_t^{\ell_1,\ell_2}f(x)=\E
        f\big(X_t^{\ell_1,\ell_2}(x)\big),\quad
        t\geq0,\,f\in\Bscr_b(\R^{m+d}), \,x\in\R^{m+d},
    $$
    where $X_t^{\ell_1,\ell_2}(x)
    =(X_t^{(1),\ell_1}(x),X_t^{(2),\ell_1,\ell_2}(x))$ is the solution
    to \eqref{changeeq} with $X_0^{\ell_1,\ell_2}=x\in\R^{m+d}$.

\begin{proposition}\label{change}
        Let $l\in(0,m/2)$ and
        assume that \textbf{\upshape(H1)} and
        \textbf{\upshape(H2)} hold.
        There is some constant
        $C=C(m,d,l)>0$ such that for any $T>0$, $x=(x^{(1)},x^{(2)}),
        y=(y^{(1)},y^{(2)})\in\R^{m+d}$, and
        $f\in\Bscr_b(\R^d)$ with $f\geq1$
        \begin{align*}
            &P_{2T}^{\ell_1,
            \ell_2}\log f(y)
            \leq\log P_{2T}^{\ell_1,
            \ell_2}f(x)
            +\frac{|x^{(1)}-y^{(1)}|^2}{2}
            \left(
            \int_0^T\lambda_r^{-2}\,
            \d\ell_1(r)
            \right)^{-1}\\
            &\qquad\quad+\frac{C\e^{2K(0,T)}}
            {\int_T^{2T}\e^{-2K(T,s)}
            \,\d\ell_2(s)}
            \Bigg\{|x^{(2)}-y^{(2)}|^2
            \left(
            \int_0^T\lambda_r^{-2}\,
            \d\ell_1(r)
            \right)
            ^{-l}
            \\
            &\quad\quad\qquad
            +\left(\left[
            |x^{(1)}|^{2(l-1)^+}+
            |y^{(1)}|^{2(l-1)^+}
            \right]
            \left(
            \int_0^T\lambda_r^{-2}\,
            \d\ell_1(r)
            \right)
            ^{-l}
            +
            \left(
            \int_0^T\lambda_r^{-2}\,
            \d\ell_1(r)
            \right)
            ^{-(l\wedge1)}
            \right)\\
            &\quad\quad\qquad\quad\times|x^{(1)}-y^{(1)}|^{2(l\wedge1)}
            \int_0^{T}\e^{-2K(0,s)}
            \,\d\ell_2(s)
            \Bigg\}.
        \end{align*}
    \end{proposition}

    Following the line of \cite{Zha13, WW14, WZ15, Den14, DS16},
    for $\varepsilon_1,\varepsilon_2\in(0,1)$, we define
    $$
        \ell_i^{\varepsilon_i}(t):=
        \frac{1}{\varepsilon_i}
        \int_t^{t+\varepsilon_i}\ell_i(s)\,\d s
        +\varepsilon_it
        =\int_0^1\ell_i(\varepsilon_is+t)\,\d s
        +\varepsilon_it,
        \quad i\in\{1,2\},\,t\geq0.
    $$
    It is clear that $\ell_i^{\varepsilon_i}$ is
    absolutely continuous and strictly increasing
    with
    \begin{equation}\label{app324}
        \ell_i^{\varepsilon_i}(t)\downarrow\ell_i(t)
        \quad\text{as $\varepsilon_i\downarrow0$}
    \end{equation}
    for all $t\geq0$. For each $i\in\{1,2\}$, denote
    by $\gamma_i^{\varepsilon_i}:
    [\ell_i^{\varepsilon_i}(0),\infty)\rightarrow[0,\infty)$
    the inverse function of $\ell_i^{\varepsilon_i}$.
    By definition, $\ell_i^{\varepsilon_i}
    (\gamma_i^{\varepsilon_i}(t))=t$ for
    $t\geq\ell_i^{\varepsilon_i}(0)$,
    $\gamma_i^{\varepsilon_i}(\ell_i^{\varepsilon_i}(t))=t$
    for $t\geq0$, and $t\mapsto\gamma_i^{\varepsilon_i}(t)$
    is absolutely continuous and strictly increasing.

    Consider the approximation
    equation for $X_t^{\ell_1^{\varepsilon_1},
    \ell_2^{\varepsilon_2}}=
    (X_t^{(1),\ell_1^{\varepsilon_1}},
    X_t^{(2),\ell_1^{\varepsilon_1},\ell_2^{\varepsilon_2}})$
    \begin{equation}\label{aceq}
        \left\{
        \begin{array}{l}
            \d X_t^{(1),\ell^{\varepsilon_1}_1}
            =\sigma_t\,\d W_{\ell_1^{\varepsilon_1}(t)
            -\ell_1^{\varepsilon_1}(0)}^{(1)},\\
            \d X_t^{(2),\ell_1^{\varepsilon_1},\ell_2^{\varepsilon_2}}
            =b(t,
            X_t^{(2),\ell_1^{\varepsilon_1},\ell_2^{\varepsilon_2}})
            \,\d t
            +|X_t^{(1),\ell_1^{\varepsilon_1}}|^l
            \,\d W_{\ell_2^{\varepsilon_2}(t)
            -\ell_2^{\varepsilon_2}(0)}^{(2)}.
        \end{array}
        \right.
    \end{equation}
    Denote by $X_t^{\ell_1^{\varepsilon_1},
    \ell_2^{\varepsilon_2}}(x)
    =(X_t^{(1),\ell_1^{\varepsilon_1}}(x),
    X_t^{(2),\ell_1^{\varepsilon_1},\ell_2^{\varepsilon_2}}(x))$ the unique non-explosive (strong) solution
    to \eqref{aceq} with
    $X_0^{\ell_1^{\varepsilon_1},
    \ell_2^{\varepsilon_2}}=x\in\R^{m+d}$, and let
    $$
        P_t^{\ell_1^{\varepsilon_1},
        \ell_2^{\varepsilon_2}}f(x)=\E
        f\big(X_t^{\ell_1^{\varepsilon_1},
        \ell_2^{\varepsilon_2}}(x)\big),\quad
        t\geq0,\,f\in\Bscr_b(\R^{m+d}), \,x\in\R^{m+d}.
    $$
    Note that \eqref{aceq} is indeed driven by Brownian
    motions and thus, as in \cite{WX14},
    the method of coupling in
    two steps and Girsanov
    transformation can be used to establish the log-Harnack
    inequality for
    $P_t^{\ell_1^{\varepsilon_1},\ell_2^{\varepsilon_2}}$.

    Observe that the regular conditional
    probability $\P(\cdot|\mathscr{F}^{(1)})$
    given $\mathscr{F}^{(1)}$ exists, where
    $\mathscr{F}^{(1)}$ is the $\sigma$-algebra
    generated by $\{W^{(1)}_t\,:\,t\geq0\}$. Let
    $\mathscr{F}^{(2)}_t$ and $\mathscr{F}_t$ be the
    $\sigma$-algebras
    generated by $\{W^{(2)}_s\,:\,0\leq s\leq t\}$
    and $\{W_s\,:\,0\leq s\leq t\}$, respectively.
    For any probability measure $\tilde{\P}$, we
    denote by $\E_{\tilde{\P}}$ the expectation
    w.r.t.\ $\tilde{\P}$. If $\tilde{\P}=\P$,
    we simply denote the expectation by $\E$
    as usual.

    \begin{lemma}\label{jhf5f}
        Fix $\varepsilon_1,\varepsilon_2\in(0,1]$,
        let $l\in(0,m/2)$, and
        assume that \textbf{\upshape(H1)}
        and \textbf{\upshape(H2)} hold.
        There is some constant
        $C=C(m,d,l)>0$ such that for any $T>0$, $x=(x^{(1)},x^{(2)}),
        y=(y^{(1)},y^{(2)})\in\R^{m+d}$, and
        $f\in\Bscr_b(\R^d)$ with $f\geq1$
        \begin{align*}
            &P_{2T}^{\ell_1^{\varepsilon_1},
            \ell_2^{\varepsilon_2}}\log f(y)
            \leq\log P_{2T}^{\ell_1^{\varepsilon_1},
            \ell_2^{\varepsilon_2}}f(x)
            +\frac{|x^{(1)}-y^{(1)}|^2}{2}
            \left(
            \int_0^T\lambda_r^{-2}\,
            \d\ell^{\varepsilon_1}_1(r)
            \right)^{-1}\\
            &\quad\quad+\frac{C\e^{2K(0,T)}}
            {\int_T^{2T}\e^{-2K(T,s)}
            \,\d\ell^{\varepsilon_2}_2(s)}
            \Bigg\{|x^{(2)}-y^{(2)}|^2
            \left(
            \int_0^T\lambda_r^{-2}\,
            \d\ell^{\varepsilon_1}_1(r)
            \right)^{-l}
            \\
            &\quad\qquad+
            \left(\left[
            |x^{(1)}|^{2(l-1)^+}+
            |y^{(1)}|^{2(l-1)^+}
            \right]
            \left(
            \int_0^T\lambda_r^{-2}\,
            \d\ell^{\varepsilon_1}_1(r)
            \right)^{-l}
            +
            \left(
            \int_0^T\lambda_r^{-2}\,
            \d\ell^{\varepsilon_1}_1(r)
            \right)^{-(l\wedge1)}
            \right)\\
            &\quad\quad\qquad\times|x^{(1)}-y^{(1)}|^{2(l\wedge1)}
            \int_0^{T}\e^{-2K(0,s)}
            \,\d\ell^{\varepsilon_2}_2(s)
            \Bigg\}.
        \end{align*}
    \end{lemma}

    \begin{proof}
        We divide the proof into four steps.

        \emph{Step 1:}
        Fix $T>0$, $x=(x^{(1)},x^{(2)}),
        y=(y^{(1)},y^{(2)})\in\R^{m+d}$ and
        let $Y_t^{(1)}$ solve the equation
        \begin{equation}\label{kkuugg}
        \left\{
        \begin{array}{l}
            \d Y_t^{(1)}=
            \sigma_t\,\d W_{\ell_1^{\varepsilon_1}(t)
            -\ell_1^{\varepsilon_1}(0)}^{(1)}
            +\xi^{(1)}_t\frac{X_t^{(1),\ell^{\varepsilon_1}_1}(x)-Y_t^{(1)}}
            {\big|X_t^{(1),\ell^{\varepsilon_1}_1}(x)-Y_t^{(1)}
            \big|}
            \I_{[0,\tau^{(1)})}(t)
            \,\d\ell_1^{\varepsilon_1}(t),\\
            Y_0^{(1)}=y^{(1)},
        \end{array}
        \right.
        \end{equation}
        where
        $$
            \xi^{(1)}_t:=\frac{|x^{(1)}-y^{(1)}|}
            {\int_0^T\lambda_r^{-2}\,
            \d\ell^{\varepsilon_1}_1(r)}
            \lambda_t^{-2}\quad \text {and} \quad
            \tau^{(1)}:=\inf\big\{t\geq0\,:\,
            X_t^{(1),\ell^{\varepsilon_1}_1}(x)=Y_t^{(1)}
            \big\}.
        $$
        Since
        $$
            \R^m\times\R^m\ni(z,z')\mapsto\I_{\{z\neq z'\}}
            \frac{z-z'}{|z-z'|}\in\R^m
        $$
        is locally Lipschitz continuous off
        the diagonal, the coupling
        $(X_t^{(1),\ell^{\varepsilon_1}_1}(x),Y_t^{(1)})$
        is well-defined and unique for $t<\tau^{(1)}$.
        If $\tau^{(1)}<\infty$, we set
        $Y_t^{(1)}=X_t^{(1),\ell^{\varepsilon_1}_1}(x)$ for
        $t\in[\tau^{(1)},\infty)$. In this way,
        we can construct a unique
        solution $(Y_t^{(1)})_{t\geq0}$
        to \eqref{kkuugg}. By the
        differential formula
        \begin{equation}\label{diffe12}
            \d|\zeta|=\I_{\{\zeta\neq 0\}}|\zeta|^{-1}
            \langle\zeta,\d\zeta\rangle,
        \end{equation}
        we have for $t<\tau^{(1)}$ that
        \begin{equation}\label{j4d2vd}
        \begin{aligned}
            \big|X_t^{(1),\ell^{\varepsilon_1}_1}(x)
            -Y_t^{(1)}\big|
            &=|x^{(1)}-y^{(1)}|
            -\int_0^t\xi^{(1)}_s\,
            \d\ell^{\varepsilon_1}_1(s)\\
            &=\left(1-
            \frac{\int_0^t\lambda_s^{-2}\,
            \d\ell^{\varepsilon_1}_1(s)}
            {\int_0^T\lambda_r^{-2}\,
            \d\ell^{\varepsilon_1}_1(r)}
            \right)
            |x^{(1)}-y^{(1)}|.
        \end{aligned}
        \end{equation}
        Then it must be $\tau^{(1)}\leq T$. Indeed,
        if $\tau^{(1)}(\omega)>T$ for
        some $\omega\in\Omega$, we can take $t=T$
        in the above equality to get
        $$
            0<|X_T^{(1),\ell^{\varepsilon_1}_1}(x)-Y_T^{(1)}|
            (\omega)=0,
        $$
        which is absurd. Let
        $$
            \widetilde{W}_t^{(1)}=
            W_t^{(1)}+\int_0^t\eta_s^{(1)}\,\d s
            \quad \text{and}\quad
            M_t^{(1)}=-\int_0^t\langle\eta_s^{(1)}
            ,\d W_s^{(1)}\rangle
            \quad \text{for $t\geq0$},
        $$
        where
        $$
            \eta_s^{(1)}:=
            \I_{[0,\tau^{(1)})}
            \big(\gamma_1^{\varepsilon_1}(s+\ell_1^{\varepsilon_1}(0))
            \big)
            \,\xi^{(1)}_
            {\gamma_1^{\varepsilon_1}(s+\ell_1^{\varepsilon_1}(0))}
            \frac{\sigma_{\gamma_1^{\varepsilon_1}
            (s+\ell_1^{\varepsilon_1}(0))}
            ^{-1}
            \big(X_{\gamma_1^{\varepsilon_1}(s+\ell_1^{\varepsilon_1}(0))}
            ^{(1),\ell^{\varepsilon_1}_1}(x)
            -Y_{\gamma_1^{\varepsilon_1}(s+\ell_1^{\varepsilon_1}(0))}
            ^{(1)}\big)}
            {\big|X_{\gamma_1^{\varepsilon_1}(s+\ell_1^{\varepsilon_1}(0))}
            ^{(1),\ell^{\varepsilon_1}_1}(x)
            -Y_{\gamma_1^{\varepsilon_1}(s+\ell_1^{\varepsilon_1}(0))}
            ^{(1)}\big|}
        $$
        for $s\geq0$.
        By \textbf{(H1)}, we have for any $s\in[0,T]$
        $$
            \big|\eta^{(1)}_{\ell_1^{\varepsilon_1}(s)
            -\ell_1^{\varepsilon_1}(0)}\big|
            \leq\xi^{(1)}_s
            \frac{|\sigma_s^{-1}
            (X_s^{(1),\ell^{\varepsilon_1}_1}(x)-Y_s^{(1)})|}
            {|X_s^{(1),\ell^{\varepsilon_1}_1}(x)-Y_s^{(1)}|}
            \leq\frac{|x^{(1)}-y^{(1)}|}
            {\int_0^T\lambda_r^{-2}\,
            \d\ell^{\varepsilon_1}_1(r)}
            \lambda_s^{-1},
        $$
        which implies that the compensator of
        the martingale $M_t$ satisfies
        $$
            \langle M^{(1)}\rangle_{
            t}
            \leq\int_0^T\big|\eta^{(1)}_{\ell_1^{\varepsilon_1}(s)
            -\ell_1^{\varepsilon_1}(0)}\big|^2\,\d
            \ell_1^{\varepsilon_1}(s)
            \leq\frac{|x^{(1)}-y^{(1)}|^2}
            {\int_0^T\lambda_r^{-2}\,
            \d\ell^{\varepsilon_1}_1(r)},\quad t\geq0.
        $$
        This, together with Novikov's criterion, yields
        that $\E R^{(1)}_t=1$, where
        $$
            R^{(1)}_t:=\exp\left[
            M^{(1)}_{t}
            -\frac12
            \langle M^{(1)}\rangle_{
            t}
            \right],\quad t\geq0.
        $$
        According to Girsanov's theorem, for
        any $t\geq0$, $(\widetilde{W}_s^{(1)})_{s\geq0}$
        is an $m$-dimensional Brownian motion under
        the new
        probability measure $R^{(1)}_t\P$. Thus,
        for all $t\geq0$,
        \begin{equation}\label{entropy1}
        \begin{aligned}
            \E\left[R^{(1)}_t\log R^{(1)}_t\right]
            &=\E_{R^{(1)}_t\P}\left[
            -\int_0^{t}
            \eta^{(1)}_s\,\d\widetilde{W}_s^{(1)}
            +\frac12
            \langle M^{(1)}\rangle_{t}
            \right]\\
            &\leq\frac{|x^{(1)}-y^{(1)}|^2}{2}
            \left(
            \int_0^T\lambda_r^{-2}\,
            \d\ell^{\varepsilon_1}_1(r)
            \right)^{-1}.
        \end{aligned}
        \end{equation}

    \emph{Step 2:}
        Consider the following SDE
        \begin{equation}\label{fdex45}
        \left\{
        \begin{array}{l}
            \d Y_t^{(2)}=
            b(t,Y_t^{(2)})\,\d t
            +|Y_t^{(1)}|^l
            \,\d W_{\ell_2^{\varepsilon_2}(t)
            -\ell_2^{\varepsilon_2}(0)}^{(2)}\\
            \qquad\qquad\qquad\qquad\quad
            +\xi_t^{(2)}
            \frac{X_t^{(2),\ell_1^{\varepsilon_1},
            \ell^{\varepsilon_2}_1}(x)-Y_t^{(2)}}
            {\big|X_t^{(2),\ell_1^{\varepsilon_1},
            \ell^{\varepsilon_2}_2}(x)-Y_t^{(2)}\big|}
            \I_{[T,\tau^{(2)})}(t)
            \,\d\ell_2^{\varepsilon_2}(t),\\
            Y_0^{(2)}=y^{(2)},
        \end{array}
        \right.
        \end{equation}
        where
        \begin{align*}
            \xi_t^{(2)}&:=\frac
            {\e^{-K(T,t)}}
            {\int_T^{2T}
            \e^{-2K(T,s)}
            \,
            \d\ell_2^{\varepsilon_2}(s)}
            \big|X_T^{(2),\ell_1^{\varepsilon_1},
            \ell^{\varepsilon_2}_2}(x)-Y_T^{(2)}\big|
            ,\quad t\geq0,\\
            \tau^{(2)}&:=\inf\big\{t\geq T\,:\,
            X_t^{(2),\ell_1^{\varepsilon_1},
            \ell^{\varepsilon_2}_2}(x)=Y_t^{(2)}
            \big\}.
        \end{align*}
        Since $Y_t^{(1)}$ is now fixed, the
        equation \eqref{fdex45} has
        a unique solution for $t<\tau^{(2)}$. Let
        $Y_t^{(2)}=X_t^{(2),\ell^{\varepsilon_2}_2}(x)$ for
        $t\in[\tau^{(2)},\infty)$. Thus, $Y_t^{(2)}$ solves
        \eqref{fdex45} for all $t\geq0$. Noting that
        $X_t^{(1),\ell_1^{\varepsilon_1}}(x)=Y_t^{(1)}$ for
        $t\geq T$, it follows from \eqref{diffe12} and \textbf{\upshape(H2)} that
        for $t\in[T,\tau^{(2)})$
        \begin{align*}
            &\big|X_t^{(2),\ell_1^{\varepsilon_1},
            \ell^{\varepsilon_2}_2}(x)-Y_t^{(2)}\big|
            \e^{-K(T,t)}
            -\big|X_T^{(2),\ell_1^{\varepsilon_1},
            \ell^{\varepsilon_2}_2}(x)-Y_T^{(2)}\big|\\
            &\quad=-\int_T^t\xi_s^{(2)}
            \e^{-K(T,s)}\,\d\ell_2^{\varepsilon_2}(s)
            -\int_T^tk(s)
            \big|X_s^{(2),\ell^{\varepsilon_2}_2}(x)
            -Y_s^{(2)}\big|
            \e^{-K(T,s)}
            \,\d s\\
            &\qquad+\int_T^t
            \frac{\big\langle
            X_s^{(2),\ell^{\varepsilon_2}_2}(x)-Y_s^{(2)},
            b(s,
            X_s^{(2),\ell^{\varepsilon_2}_2}(x))
            -b(s,Y_s^{(2)})
            \big\rangle}
            {|X_s^{(2),\ell^{\varepsilon_2}_2}-Y_s^{(2)}|}
            \e^{-K(T,s)}
            \,\d s\\
            &\quad\leq
            -\int_T^t\xi_s^{(2)}
            \e^{-K(T,s)}\,\d\ell_2^{\varepsilon_2}(s).
        \end{align*}
        Similarly as in the first part of the proof, it is easy to see that
        this implies $\tau^{(2)}\leq 2T$.
        For $t\geq0$ and $n\in\N$, let
        \begin{align*}
            &\eta_t^{(2)}=
            \I_{[T,\tau^{(2)})}
            \big(\gamma_2^{\varepsilon_2}(t+\ell_2^{\varepsilon_2}(0))
            \big)
            \,\xi^{(2)}_{
            \gamma_2^{\varepsilon_2}
            (t+\ell_2^{\varepsilon_2}(0))}
            |Y_{\gamma_2^{\varepsilon_2}(t+\ell_2^{\varepsilon_2}(0))}
            ^{(1)}|^{-l}
            \frac{X_{\gamma_2^{\varepsilon_2}(t+\ell_2^{\varepsilon_2}(0))}
            ^{(2),\ell_1^{\varepsilon_1},\ell^{\varepsilon_2}_2}(x)
            -Y_{\gamma_2^{\varepsilon_2}(t+\ell_2^{\varepsilon_2}(0))}
            ^{(1)}}
            {\big|X_{\gamma_2^{\varepsilon_2}(t+\ell_2^{\varepsilon_2}(0))}
            ^{(2),\ell_1^{\varepsilon_1},\ell^{\varepsilon_2}_2}(x)
            -Y_{\gamma_2^{\varepsilon_2}(t+\ell_2^{\varepsilon_2}(0))}
            ^{(1)}\big|},\\
            &\eta_t^{(2)}(n)=\eta_t^{(2)}
            \I_{\{|\eta_t^{(2)}|\leq n\}},\\
            &R^{(2)}_t=\exp\left[
            -\int_{\ell_2^{\varepsilon_2}(T)
            -\ell_2^{\varepsilon_2}(0)}^t
            \langle\eta_s^{(2)},\d W_s^{(2)}\rangle
            -\frac12\int_{\ell_2^{\varepsilon_2}(T)
            -\ell_2^{\varepsilon_2}(0)}^t
            |\eta_s^{(2)}|^2\,\d s
            \right],\\
            &R^{(2)}_t(n)=\exp\left[
            -\int_{\ell_2^{\varepsilon_2}(T)
            -\ell_2^{\varepsilon_2}(0)}^t
            \langle\eta_s^{(2)}(n),\d W_s^{(2)}\rangle
            -\frac12\int_{\ell_2^{\varepsilon_2}(T)
            -\ell_2^{\varepsilon_2}(0)}^t
            |\eta_s^{(2)}(n)|^2\,\d s
            \right].
        \end{align*}
        By Girsanov's theorem, under the weighted probability
        measure $R^{(2)}_{
        \ell_2^{\varepsilon_2}(2T)-
        \ell_2^{\varepsilon_2}(0)}(n)\P(\cdot|\mathscr{F}^{(1)})$,
        the process
        $$
            \widetilde{W}_t^{(2),n}:=
            W_t^{(2)}+\int_{
            \ell_2^{\varepsilon_2}(T)
            -\ell_2^{\varepsilon_2}(0)}^{t\vee
            [\ell_2^{\varepsilon_2}(T)
            -\ell_2^{\varepsilon_2}(0)]}
            \eta_s^{(2)}(n)\,\d s,
            \quad t\geq0,
        $$
        is a standard $d$-dimensional Brownian motion. Then we have
        for all $n\in\N$ and
        $t\geq0$
        \begin{align*}
            \E_{\P(\cdot|\mathscr{F}^{(1)})}\big[R^{(2)}_t(n)
            \log R^{(2)}_t(n)\big]
            &=\E_{R^{(2)}_t(n)\P(\cdot|\mathscr{F}^{(1)})}\big[
            \log R^{(2)}_t(n)\big]\\
            &=\frac12\E_{R^{(2)}_t(n)\P(\cdot|\mathscr{F}^{(1)})}\left[
            \int_{\ell_2^{\varepsilon_2}(T)
            -\ell_2^{\varepsilon_2}(0)}^t
            |\eta_s^{(2)}(n)|^2\,\d s
            \right]\\
            &\leq\frac12\E_{R^{(2)}_t(n)\P(\cdot|\mathscr{F}^{(1)})}\left[
            \int_{T}^{2T}
            |\eta_{\ell_2^{\varepsilon_2}(s)-
            \ell_2^{\varepsilon_2}(0)}^{(2)}(n)|^2\,\d
            \ell_2^{\varepsilon_2}(s)
            \right]\\
            &\leq\frac12\E_{\P(\cdot|\mathscr{F}^{(1)})}\left[
            R^{(2)}_t(n)\int_{T}^{2T}
            |\xi_s^{(2)}|^2|Y_s^{(1)}|^{-2l}\,\d
            \ell_2^{\varepsilon_2}(s)
            \right]\\
            &=\frac{\E_{\P(\cdot|\mathscr{F}^{(1)})}\left[
            R^{(2)}_t(n)\big|X_T^{(2),\ell_1^{\varepsilon_1},
            \ell_2^{\varepsilon_2}}(x)
            -Y_T^{(2)}\big|^2
            \right]}
            {2\left(
            \int_T^{2T}
            \e^{-2K(T,s)}
            \,\d\ell_2
            ^{\varepsilon_2}(s)
            \right)^2}\\
            &\qquad\times\int_T^{2T}
            \e^{-2K(T,s)}
            |Y_s^{(1)}|^{-2l}
            \,\d\ell_2
            ^{\varepsilon_2}(s)\\
            &=\frac{\E_{\P(\cdot|\mathscr{F}^{(1)})}
            \big|X_T^{(2),\ell_1^{\varepsilon_1},
            \ell_2^{\varepsilon_2}}(x)
            -Y_T^{(2)}\big|^2
            }
            {2\left(
            \int_T^{2T}\e^{-2K(T,s)}\,\d\ell_2
            ^{\varepsilon_2}(s)
            \right)^2}\\
            &\qquad\times\int_T^{2T}\e^{-2K(T,s)}
            |Y_s^{(1)}|^{-2l}
            \,\d\ell_2
            ^{\varepsilon_2}(s),
        \end{align*}
        where in the last equality we have used the fact that,
        for $t\geq0$, $R^{(2)}_t(n)$ is an $\mathscr{F}_t^{(2)}$-martingale
        under $\P(\cdot|\mathscr{F}^{(1)})$ and
        $R^{(2)}_{\ell_2^{\varepsilon_2}(T)
        -\ell_2^{\varepsilon_2}(0)}(n)=1$. By
        It\^{o}'s formula, \textbf{(H2)} and the inequality
        $$
            \left(|u|^l-|v|^l\right)^2\leq(l\vee1)^2
            |u-v|^{2(l\wedge1)}\left(|v-u|+|v|\right)
            ^{2(l-1)^+},\quad u,v\in\R^d,
        $$
        we get that for $t\in[0,T]$
        \begin{align*}
            &\d\big|X_t^{(2),\ell_1^{\varepsilon_1},
            \ell_2^{\varepsilon_2}}(x)
            -Y_t^{(2)}\big|^2\\
            &=2\big\langle X_t^{(2),\ell_1^{\varepsilon_1},
            \ell_2^{\varepsilon_2}}(x)
            -Y_t^{(2)},b(t,X_t^{(2),\ell_1^{\varepsilon_1},
            \ell_2^{\varepsilon_2}}(x))
            -b(t,Y_t^{(2)})
            \big\rangle\,\d t\\
            &\quad+2\big(|X_t^{(1),\ell_1^{\varepsilon_1}}(x)|^l
            -|Y_t^{(1)}|^l\big)
            \big\langle X_t^{(2),\ell_1^{\varepsilon_1},
            \ell_2^{\varepsilon_2}}(x)
            -Y_t^{(2)},
            \d W^{(2)}_{\ell_2^{\varepsilon_2}(t)
            -\ell_2^{\varepsilon_2}(0)}\big\rangle\\
            &\quad+d\big(|X_t^{(1),\ell_1^{\varepsilon_1}}(x)|^l
            -|Y_t^{(1)}|^l\big)^2\,\d\ell_2^{\varepsilon_2}(t)\\
            &\leq2k(t)\big|X_t^{(2),\ell_1^{\varepsilon_1},
            \ell_2^{\varepsilon_2}}(x)
            -Y_t^{(2)}\big|^2\,\d t\\
            &\quad+2\big(|X_t^{(1),\ell_1^{\varepsilon_1}}(x)|^l
            -|Y_t^{(1)}|^l\big)
            \big\langle X_t^{(2),\ell_1^{\varepsilon_1},
            \ell_2^{\varepsilon_2}}(x)
            -Y_t^{(2)},
            \d W^{(2)}_{\ell_2^{\varepsilon_2}(t)
            -\ell_2^{\varepsilon_2}(0)}\big\rangle\\
            &\quad+d(l\vee1)^2
            \big|X_t^{(1),\ell_1^{\varepsilon_1}}(x)
            -Y_t^{(1)}\big|^{2(l\wedge1)}
            \big(
            \big|Y_t^{(1)}-X_t^{(1),\ell_1^{\varepsilon_1}}(x)
            \big|
            +|Y_t^{(1)}|
            \big)^{2(l-1)^+}
            \,\d\ell_2^{\varepsilon_2}(t).
        \end{align*}
        Since it follows from \eqref{j4d2vd} that
        $\big|X_t^{(1),\ell_1^{\varepsilon_1}}(x)
            -Y_t^{(1)}\big|\leq|x^{(1)}-y^{(1)}|$,
       this implies that for $t\in[0,T]$
        \begin{align*}
            &\e^{-2K(0,T)}
            \E_{\P(\cdot|\mathscr{F}^{(1)})}
            \big|X_T^{(2),\ell_1^{\varepsilon_1},
            \ell_2^{\varepsilon_2}}(x)
            -Y_T^{(2)}\big|^2-
            |x^{(2)}-y^{(2)}|^2\\
            &\leq
            d(l\vee1)^2
            \int_0^T
            \big|X_t^{(1),\ell_1^{\varepsilon_1}}(x)
            -Y_t^{(1)}\big|^{2(l\wedge1)}
            \big(
            \big|Y_t^{(1)}-X_t^{(1),\ell_1^{\varepsilon_1}}(x)\big|
            +|Y_t^{(1)}|
            \big)^{2(l-1)^+}
            \e^{-2K(0,t)}
            \,\d\ell_2^{\varepsilon_2}(t)\\
            &\leq
            d(l\vee1)^2
            |x^{(1)}-y^{(1)}|^{2(l\wedge1)}
            \int_0^T
            \big(
            |x^{(1)}-y^{(1)}|
            +|Y_t^{(1)}|
            \big)^{2(l-1)^+}
            \e^{-2K(0,t)}
            \,\d\ell_2^{\varepsilon_2}(t).
        \end{align*}
        Now we know
        that for all $n\in\N$ and
        $t\geq0$
        \begin{equation}\label{entropy2n}
        \begin{aligned}
            &\E_{\P(\cdot|\mathscr{F}^{(1)})}\big[R^{(2)}_t(n)
            \log R^{(2)}_t(n)\big]\cdot
            \frac{2\left(
            \int_T^{2T}\e^{
            -2K(T,s)
            }\,\d\ell_2
            ^{\varepsilon_2}(s)
            \right)^2}
            {\e^{2K(0,T)}
            \int_T^{2T}
            \e^{-2K(T,s)}
            |Y_s^{(1)}|^{-2l}\,\d\ell_2
            ^{\varepsilon_2}(s)
            }\leq|x^{(2)}-y^{(2)}|^2
            \\
            &\qquad
            +
            d(l\vee1)^2
            |x^{(1)}-y^{(1)}|^{2(l\wedge1)}
            \int_0^T
            \big(
            |x^{(1)}-y^{(1)}|
            +|Y_s^{(1)}|
            \big)^{2(l-1)^+}
            \e^{
            -2K(0,s)
            }\,\d\ell_2
            ^{\varepsilon_2}(s).
        \end{aligned}
        \end{equation}

    \emph{Step 3:}
        For $t\geq0$ and $n\in\N$,
        let $R_t(n)=R^{(1)}_t\cdot R^{(2)}_t(n)$
        and $R_t=R^{(1)}_t\cdot R^{(2)}_t$.
        Since for any $t\geq0$ and $s\geq0$,
        the distribution of $Y_s^{(1)}$ under
        $R^{(1)}_t\P$ coincides with that of
        $X_s^{(1),\ell_1^{\varepsilon_1}}(x)$ under $\P$,
        it follows from \eqref{entropy1}
        and \eqref{entropy2n} that
        \begin{equation}\label{dfs3w1x4}
        \begin{aligned}
            &\E[R_t(n)\log R_t(n)]\\
            &\quad\quad=\E\big\{R^{(1)}_t\log R^{(1)}_t
            \E_{\P(\cdot|\mathscr{F}^{(1)})}
            R^{(2)}_t(n)
            \big\}
            +\E\big\{
            R^{(1)}_t\E_{\P(\cdot|\mathscr{F}^{(1)})}
            [R^{(2)}_t(n)\log R^{(2)}_t(n)]
            \big\}\\
            &\quad\quad=\E[R^{(1)}_t\log R^{(1)}_t]
            +\E_{R^{(1)}_t\P}
            \big\{
            \E_{\P(\cdot|\mathscr{F}^{(1)})}
            [R^{(2)}_t(n)\log R^{(2)}_t(n)]
            \big\}\\
            &\quad\quad\leq\frac{|x^{(1)}-y^{(1)}|^2}{2}
            \left(
            \int_0^T\lambda_r^{-2}\,
            \d\ell^{\varepsilon_1}_1(r)
            \right)^{-1}+\frac{\e^{2K(0,T)}}{2\left(
            \int_T^{2T}\e^{-2K(T,s)}
            \,\d\ell^{\varepsilon_2}_2(s)
            \right)^2}\\
            &\quad\quad\quad\times\Big(
            |x^{(2)}-y^{(2)}|^2
            \E I_1
            +d(l\vee1)^2
            |x^{(1)}-y^{(1)}|^{2(l\wedge1)}
            \E[I_1I_2]
            \Big),
        \end{aligned}
        \end{equation}
        where
        $$
            I_1:=\int_T^{2T}\e^{-2K(T,s)}
            |X_s^{(1),\ell_1^{\varepsilon_1}}(x)|
            ^{-2l}\,\d\ell^{\varepsilon_2}_2(s)
        $$
        and
        $$
            I_2:=\int_0^T\e^{-2K(0,s)}
            \big(|x^{(1)}-y^{(1)}|
            +|X_s^{(1),\ell_1^{\varepsilon_1}}(x)|
            \big)^{2(l-1)^+}
            \,\d\ell^{\varepsilon_2}_2(s).
        $$
        It follows from the elementary inequality
        \begin{equation}\label{jh5dc4}
            \left(\sum_{i=1}^na_i\right)^r
            \leq n^{(r-1)^+}
            \sum_{i=1}^na_i^r,
            \quad n\in\N,\,a_i\geq0,\,r\geq0
        \end{equation}
        that
        \begin{align*}
            I_2&\leq\int_0^T\e^{-2K(0,s)}
            \left(|x^{(1)}-y^{(1)}|+|x^{(1)}|
            +\left|\int_0^s\sigma_r\,
            \d W^{(1)}_
            {\ell_1^{\varepsilon_1}(s)-
            \ell_1^{\varepsilon_1}(0)}
            \right|
            \right)^{2(l-1)^+}
            \,\d\ell^{\varepsilon_2}_2(s)\\
            &\leq3^{(2l-3)^+}\left(
            |x^{(1)}-y^{(1)}|^{2(l-1)^+}
            +|x^{(1)}|^{2(l-1)^+}
            \right)\int_0^T\e^{-2K(0,s)}
            \,\d\ell^{\varepsilon_2}_2(s)
            +3^{(2l-3)^+}I_3\\
            &\leq3^{(2l-3)^+}2^{1+(2l-3)^+}\left(
            |x^{(1)}|^{2(l-1)^+}+
            |y^{(1)}|^{2(l-1)^+}
            \right)\int_0^T\e^{-2K(0,s)}
            \,\d\ell^{\varepsilon_2}_2(s)
            +3^{(2l-3)^+}I_3,
        \end{align*}
        where
        $$
            I_3:=\int_0^T\e^{-2K(0,s)}
            \left|\int_0^s\sigma_r\,
            \d W^{(1)}_
            {\ell_1^{\varepsilon_1}(s)-
            \ell_1^{\varepsilon_1}(0)}
            \right|
            ^{2(l-1)^+}
            \,\d\ell^{\varepsilon_2}_2(s).
        $$
        According to Lemma \ref{ahf2} in the
        appendix, for any $\theta\in(0,m/2)$,
        there exists $c=c(m,\theta)>0$ such that
        \begin{equation}\label{j11g6fcd}
        \begin{aligned}
            \E|X_s^{(1),\ell_1^{\varepsilon_1}}(x)|^{-2\theta}
            &=\E\left|x^{(1)}+\int_0^s\sigma_r\,
            \d W^{(1)}_
            {\ell_1^{\varepsilon_1}(s)-
            \ell_1^{\varepsilon_1}(0)}
            \right|
            ^{-2\theta}\\
            &\leq c\left(
            \int_0^s\|\sigma_r\|^2\,
            \d\ell^{\varepsilon_1}_1(r)
            \right)^{-\theta}\\
            &\leq c\left(
            \int_0^s\lambda_r^{-2}\,
            \d\ell^{\varepsilon_1}_1(r)
            \right)^{-\theta}
        \end{aligned}
        \end{equation}
        holds for all $s>0$ and $x\in\R^{m+d}$.
        This yields that for some
        $c_1=c_1(m,l)>0$
        \begin{align*}
            \E I_1&\leq c_1
            \int_T^{2T}\e^{-2K(T,s)}
            \left(
            \int_0^s\lambda_r^{-2}\,
            \d\ell^{\varepsilon_1}_1(r)
            \right)^{-l}
            \,\d\ell^{\varepsilon_2}_2(s)\\
            &\leq c_1
            \left(
            \int_0^T\lambda_r^{-2}\,
            \d\ell^{\varepsilon_1}_1(r)
            \right)^{-l}
            \int_T^{2T}\e^{-2K(T,s)}
            \,\d\ell^{\varepsilon_2}_2(s).
        \end{align*}
        Since $l<m/2$, we can pick $p=p(l)>1$
        such that $pl<m/2$. It follows from
        the H\"{o}lder inequality
        and \eqref{j11g6fcd} that
        \begin{align*}
            \E I_1^p&=\E\left(
            \int_T^{2T}\e^{-\frac{2(p-1)}{p}K(T,s)}\cdot
            \e^{-\frac2pK(T,s)}
            |X_s^{(1),\ell_1^{\varepsilon_1}}(x)|
            ^{-2l}\,\d\ell^{\varepsilon_2}_2(s)
            \right)^p\\
            &\leq\left(\int_T^{2T}\e^{-2K(T,s)}
            \,\d\ell^{\varepsilon_2}_2(s)\right)^{p-1}
            \cdot
            \E\left[
            \int_T^{2T}\e^{-2K(T,s)}
            |X_s^{(1),\ell_1^{\varepsilon_1}}(x)|
            ^{-2pl}\,\d\ell^{\varepsilon_2}_2(s)
            \right]\\
            &\leq\left(\int_T^{2T}\e^{-2K(T,s)}
            \,\d\ell^{\varepsilon_2}_2(s)\right)^{p-1}
            \cdot
            c_2
            \int_T^{2T}\e^{-2K(T,s)}
            \left(
            \int_0^s\lambda_r^{-2}\,
            \d\ell^{\varepsilon_1}_1(r)
            \right)^{-pl}
            \,\d\ell^{\varepsilon_2}_2(s)
            \\
            &\leq c_2\left(
            \int_0^T\lambda_r^{-2}\,
            \d\ell^{\varepsilon_1}_1(r)
            \right)^{-pl}
            \left(
            \int_T^{2T}\e^{-2K(T,s)}
            \,\d\ell^{\varepsilon_2}_2(s)
            \right)^p
        \end{align*}
        for some $c_2=c_2(m,p,l)>0$. Moreover, we have
        \begin{align*}
            &\E I_3^{\frac{p}{p-1}}
            =\E\left(
            \int_0^{T}\e^{-\frac2pK(0,s)}\cdot
            \e^{-\frac{2(p-1)}{p}K(0,s)}
            \left|\int_0^s\sigma_r\,
            \d W^{(1)}_
            {\ell_1^{\varepsilon_1}(s)-
            \ell_1^{\varepsilon_1}(0)}
            \right|^{2(l-1)^+}\,\d\ell^{\varepsilon_2}_2(s)
            \right)^{\frac{p}{p-1}}\\
            &\leq\left(\int_0^{T}\e^{-2K(0,s)}
            \,\d\ell^{\varepsilon_2}_2(s)\right)
            ^{\frac{1}{p-1}}
            \cdot
            \E\left[
            \int_0^{T}\e^{-2K(0,s)}
            \left|\int_0^s\sigma_r\,
            \d W^{(1)}_
            {\ell_1^{\varepsilon_1}(s)-
            \ell_1^{\varepsilon_1}(0)}
            \right|
            ^{\frac{2p}{p-1}(l-1)^+}
            \d\ell^{\varepsilon_2}_2(s)
            \right]\\
            &\leq\left(\int_0^{T}\e^{-2K(0,s)}
            \,\d\ell^{\varepsilon_2}_2(s)\right)^{\frac{1}{p-1}}
            \cdot
            c_3
            \int_0^{T}\e^{-2K(0,s)}
            \left(
            \int_0^s\lambda_r^{-2}\,
            \d\ell^{\varepsilon_1}_1(r)
            \right)^{\frac{p}{p-1}(l-1)^+}
            \,\d\ell^{\varepsilon_2}_2(s)
            \\
            &\leq c_3\left(
            \int_0^T\lambda_r^{-2}\,
            \d\ell^{\varepsilon_1}_1(r)
            \right)^{\frac{p}{p-1}(l-1)^+}
            \left(
            \int_0^{T}\e^{-2K(0,s)}
            \,\d\ell^{\varepsilon_2}_2(s)
            \right)^{\frac{p}{p-1}}
        \end{align*}
        for some $c_3=c_3(m,p,l)>0$. Thus,
        \begin{align*}
            \E[I_1I_3]
            &\leq\left(\E I_1^p\right)^{\frac1p}
            \left(\E I_3^{\frac{p}{p-1}}\right)
            ^{\frac{p-1}{p}}
            \\
            &\leq c_2^{\frac{1}{p}}
            c_3^{\frac{p-1}{p}}
            \left(
            \int_0^T\lambda_r^{-2}\,
            \d\ell^{\varepsilon_1}_1(r)
            \right)^{-(l\wedge1)}
            \int_0^{T}\e^{-2K(0,s)}
            \,\d\ell^{\varepsilon_2}_2(s)
            \cdot
            \int_T^{2T}\e^{-2K(T,s)}
            \,\d\ell^{\varepsilon_2}_2(s).
        \end{align*}
        Combining the above estimates,
        we get that for some positive constant
        $C=C(m,d,l)$
        \begin{align*}
            &|x^{(2)}-y^{(2)}|^2
            \E I_1
            +d(l\vee1)^2
            |x^{(1)}-y^{(1)}|^{2(l\wedge1)}
            \E[I_1I_2]\\
            &\leq
            2C|x^{(2)}-y^{(2)}|^2
            \left(
            \int_0^T\lambda_r^{-2}\,
            \d\ell^{\varepsilon_1}_1(r)
            \right)^{-l}
            \int_T^{2T}\e^{-2K(T,s)}
            \,\d\ell^{\varepsilon_2}_2(s)\\
            &\quad+2C
            \left(\left[
            |x^{(1)}|^{2(l-1)^+}+
            |y^{(1)}|^{2(l-1)^+}
            \right]
            \left(
            \int_0^T\lambda_r^{-2}\,
            \d\ell^{\varepsilon_1}_1(r)
            \right)^{-l}
            +
            \left(
            \int_0^T\lambda_r^{-2}\,
            \d\ell^{\varepsilon_1}_1(r)
            \right)^{-(l\wedge1)}
            \right)\\
            &\qquad\times
            |x^{(1)}-y^{(1)}|^{2(l\wedge1)}
            \int_0^{T}\e^{-2K(0,s)}
            \,\d\ell^{\varepsilon_2}_2(s)
            \cdot
            \int_T^{2T}\e^{-2K(T,s)}
            \,\d\ell^{\varepsilon_2}_2(s).
        \end{align*}
        This, together with \eqref{dfs3w1x4},
        gives that for all $n\in\N$
        and $t\geq0$
        \begin{equation}\label{dews231s}
        \begin{aligned}
            &\E\big[R_t(n)\log R_t(n)\big]\leq
            \frac{|x^{(1)}-y^{(1)}|^2}{2}
            \left(
            \int_0^T\lambda_r^{-2}\,
            \d\ell^{\varepsilon_1}_1(r)
            \right)^{-1}\\
            &+\frac{C\e^{2K(0,T)}}
            {\int_T^{2T}\e^{-2K(T,s)}
            \,\d\ell^{\varepsilon_2}_2(s)}
            \Bigg\{|x^{(2)}-y^{(2)}|^2
            \left(
            \int_0^T\lambda_r^{-2}\,
            \d\ell^{\varepsilon_1}_1(r)
            \right)^{-l}
            \\
            &\quad+
            \left(\left[
            |x^{(1)}|^{2(l-1)^+}+
            |y^{(1)}|^{2(l-1)^+}
            \right]
            \left(
            \int_0^T\lambda_r^{-2}\,
            \d\ell^{\varepsilon_1}_1(r)
            \right)^{-l}
            +
            \left(
            \int_0^T\lambda_r^{-2}\,
            \d\ell^{\varepsilon_1}_1(r)
            \right)^{-(l\wedge1)}
            \right)\\
            &\qquad\times
            |x^{(1)}-y^{(1)}|^{2(l\wedge1)}
            \int_0^{T}\e^{-2K(0,s)}
            \,\d\ell^{\varepsilon_2}_2(s)
            \Bigg\}.
        \end{aligned}
        \end{equation}
        It is not hard to verify that
        this implies that $(R_t)_{t\geq0}$ is an
        $\mathscr{F}_t$-martingale under $\P$,
        and thus $\E R=1$, where
        $$
            R:=R_{[\ell_1^{\epsilon_1}(T)
            -\ell_1^{\epsilon_1}(0)]
            \vee[\ell_2^{\varepsilon_2}(2T)
            -\ell_2^{\varepsilon_2}(0)]}.
        $$
        Since for any $t\geq0$, $R_t(n)\rightarrow R_t$
        as $n\rightarrow\infty$, we can
        let $n\rightarrow\infty$ in \eqref{dews231s}
        and use Fatou's lemma to know that
        \eqref{dews231s} holds with
        $R_t(n)$ replaced by $R$.

    \emph{Step 4:}
        By the Jensen inequality, we have
        for any random variable $F\geq1$,
        $$
            \E\left[R\log\frac{F}{R}\right]
            =\E_{R\P}\left[\log\frac{F}{R}\right]
            \leq\log\E_{R\P}\left[\frac{F}{R}\right]
            =\log\E F,
        $$
        hence
        \begin{equation}\label{young}
            \E\left[R\log F\right]
            \leq\log\E F + \E\left[R\log R\right].
        \end{equation}
        Let
        $$
            \widetilde{W}_t^{(2)}=
            W_t^{(2)}+\int_{
            \ell_2^{\varepsilon_2}(T)
            -\ell_2^{\varepsilon_2}(0)}^{t\vee
            [\ell_2^{\varepsilon_2}(T)
            -\ell_2^{\varepsilon_2}(0)]}
            \eta_s^{(2)}\,\d s
            =W_t^{(2)}+\int_0^t
            \eta_s^{(2)}\,\d s,
            \quad t\geq0.
        $$
        Then
        $$
            \widetilde{W}_t
            :=\big(\widetilde{W}_t^{(1)},
            \widetilde{W}_t^{(2)}\big)
            =W_t+\int_0^t\Theta_s\,\d s,\quad t\geq0,
        $$
        where
        $$
            \Theta_s:=\big(
            \eta^{(1)}_s\,,\,
            \eta^{(2)}_s
            \big).
        $$
        Clearly, we can rewrite $R_t$ as
        $$
            R_t=\exp\left[
            -\int_0^{t}
            \langle\Theta_s,\d W_s\rangle
            -\frac12\int_0^{t}
            |\Theta_s|^2\,\d s
            \right],\quad t\geq0.
        $$
        It follows from Girsanov's theorem that
        under the
        weighted probability measure $R\P$,
        $\widetilde{W}_t$ is a standard Brownian
        motion on $\R^{m+d}$.
        Noting that $Y_t=(Y_t^{(1)},Y_t^{(2)})$ solves
        the SDE
        $$
        \left\{
        \begin{array}{l}
            \d Y_t^{(1)}=
            \sigma_t\,\d \widetilde{W}_{\ell_1^{\varepsilon_1}(t)
            -\ell_1^{\varepsilon_1}(0)}^{(1)},\\
            Y_0^{(1)}=y^{(1)},\\
            \d Y_t^{(2)}=
            b(t,Y_t^{(2)})\,\d t
            +|Y_t^{(1)}|^l
            \,\d \widetilde{W}_{\ell_2^{\varepsilon_2}(t)
            -\ell_2^{\varepsilon_2}(0)}^{(2)},\\
            Y_0^{(2)}=y^{(2)},
        \end{array}
        \right.
        $$
        we conclude
        that the distribution of $Y_{2T}$ under $R\P$
        coincides with that of
        $X_{2T}^{\ell_1^{\varepsilon_1},
        \ell_2^{\varepsilon_2}}(y)$
        under $\P$. Therefore, we obtain from
        $Y_{2T}=X_{2T}^{\ell_1^{\varepsilon_1},
        \ell_2^{\varepsilon_2}}(x)$
        and \eqref{young} that for
        any $f\in\mathscr{B}_b(\R^{m+d})$ with $f\geq1$
        \begin{align*}
            P_{2T}^{\ell_1^{\varepsilon_1},
            \ell_2^{\varepsilon_2}}\log f(y)
            &=\E\log f\big(X_{2T}^{\ell_1^{\varepsilon_1},
            \ell_2^{\varepsilon_2}}(y)\big)\\
            &=\E_{R\P}\log f(Y_{2T})\\
            &=\E\big[
            R\log f\big(X_{2T}^{\ell_1^{\varepsilon_1},
            \ell_2^{\varepsilon_2}}(x)\big)
            \big]\\
            &\leq\log\E f\big(X_{2T}^{\ell_1^{\varepsilon_1},
            \ell_2^{\varepsilon_2}}(x)\big)+\E[R\log R]\\
            &=\log P_{2T}^{\ell_1^{\varepsilon_1},
            \ell_2^{\varepsilon_2}}f(x)
            +\E[R\log R].
        \end{align*}
        Inserting the estimate \eqref{dews231s} with
        $R_t(n)$ replaced by $R$ into this inequality,
        we complete the proof
        of the log-Harnack inequality.
    \end{proof}

To prove Proposition \ref{change} by using Lemma \ref{jhf5f}, we
need some preparations. First, the
following Burkholder-Davis-Gundy type inequality
is essentially due to \cite[Lemma 2.3]{Kus09}. For
the reader's convenience, we include a simple proof.

\begin{lemma}\label{BDG}
    Let $\varrho:[0,\infty)\rightarrow\R$ be a non-decreasing
    c\`{a}dl\`{a}g function with $\varrho(0)=0$.
    For any $p>0$, there exists a constant $C_p>0$
    depending only on $p$ such that for any $t>0$
    $$
        \E\left[\sup_{s\in[0,t]}\left|\int_0^s\sigma_r\,
        \d W^{(1)}_{\varrho(r)}\right|^p\right]
        \leq C_p\left(\int_0^t\|\sigma_r\|^2\,
        \d\varrho(r)\right)^{p/2}.
    $$
\end{lemma}

\begin{proof}
    Fix $t>0$. Let $g:[0,t]\rightarrow\R$ be a bounded
    measurable function. By
    \cite[Lemma 4.2]{Zha14}, one has
    $$
        \int_0^tg(r)\,\d\varrho(r)
        =\int_0^{\varrho(t)}g\big(\varrho^{-1}(u)\big)\,
        \d u,
    $$
    where $\varrho^{-1}(u):=\inf\{s\geq0\,:\,\varrho(s)>u\}$.
    A similar argument shows that
    $$
        \int_0^tg(r)\,\d W^{(1)}_{\varrho(r)}
        =\int_0^{\varrho(t)}g\big(\varrho^{-1}(u)\big)\,
        \d W^{(1)}_u.
    $$
    Then by the Burkholder-Davis-Gundy inequality, we get
    \begin{align*}
        \E\left[\sup_{s\in[0,t]}\left|\int_0^s\sigma_r\,
        \d W^{(1)}_{\varrho(r)}\right|^p\right]
        &=\E\left[\sup_{s\in[0,t]}\left|\int_0^{\varrho(s)}
        \sigma_{\varrho^{-1}(u)}\,
        \d W^{(1)}_{u}\right|^p\right]\\
        &\leq\E\left[\sup_{s\in[0,\varrho(t)]}\left|\int_0^{s}
        \sigma_{\varrho^{-1}(u)}\,
        \d W^{(1)}_{u}\right|^p\right]\\
        &\leq C_p\left(
        \int_0^{\varrho(t)}
        \big\|\sigma_{\varrho^{-1}(u)}\big\|^2\,\d u
        \right)^{p/2}\\
        &=C_p\left(
        \int_0^t
        \|\sigma_r\|^2\,\d\varrho(r)
        \right)^{p/2}.
        \qedhere
    \end{align*}
\end{proof}

The following two assumptions will be used:
\begin{enumerate}
        \item[\textbf{(A1)}]
            $\sigma$ is piecewise constant, i.e.\ there
            exists a sequence $\{t_n\}_{n\geq0}$
            with $t_0=0$ and $t_n\uparrow\infty$
            such that
            $$
                \sigma_t=\sum_{n=1}^\infty
                \I_{[t_{n-1},t_n)}(t)
                \sigma_{t_{n-1}};
            $$

        \item[\textbf{(A2)}]
            $b_t:\R^d\rightarrow\R^d$
            is, uniformly for $t$ in compact intervals, global Lipschitz, i.e.\ for any $t>0$,
            there is some $C_t>0$ such that
            $$
                |b_s(x^{(2)})-b_s(y^{(2)})|
                \leq C_t|x^{(2)}-y^{(2)}|,
                \quad 0\leq s\leq t,\,
                x^{(2)},y^{(2)}\in\R^d.
            $$
    \end{enumerate}

\begin{lemma}\label{kjh5fc3}
    Assume \textbf{\upshape(A1)}. Then
    for any $l>0$, $x\in\R^{m+d}$
    and $t>0$,
    \begin{equation}\label{kj65fh}
        \lim_{\varepsilon_1\downarrow0}\E
        \big|X_t^{(1),\ell_1^{\varepsilon_1}}(x)
        -X_t^{(1),\ell_1}(x)\big|^2=0,
    \end{equation}
    \begin{equation}\label{kj23sdc}
        \lim_{\varepsilon_1\downarrow0}
        \limsup_{\varepsilon_2\downarrow0}
        \E\left|
        \int_0^t|X_s^{(1),\ell_1^{\varepsilon_1}}(x)|^l\,
        \d W^{(2)}_{\ell_2^{\varepsilon_2}(s)
        -\ell_2^{\varepsilon_2}(0)}
        -\int_0^t|X_s^{(1),\ell_1}(x)|^l\,
        \d W^{(2)}_{\ell_2(s)}
        \right|^2
        =0.
    \end{equation}
\end{lemma}

\begin{proof}
    Fix $l>0$, $x\in\R^{m+d}$ and $t>0$. It is not
    hard to obtain from \textbf{(A1)} that
    \begin{align*}
        &\big|X_t^{(1),\ell_1^{\varepsilon_1}}(x)
        -X_t^{(1),\ell_1}(x)\big|\\
        &\qquad\leq\sup_{s\in[0,t]}\|\sigma_s\|
        \left(
        \big|W^{(1)}_{\ell_1^{\varepsilon_1}(t)-
        \ell_1^{\varepsilon_1}(0)}
        -W^{(1)}_{\ell_1(t)}\big|
        +2\sum_{n:\,t_n<t}
        \big|W^{(1)}_{\ell_1^{\varepsilon_1}(t_n)-
        \ell_1^{\varepsilon_1}(0)}
        -W^{(1)}_{\ell_1(t_n)}\big|
        \right),
    \end{align*}
    which, together with \eqref{app324},
    implies \eqref{kj65fh}.

    By the isometry property of
    stochastic integrals, we have
    \begin{equation}\label{df3wa}
    \begin{aligned}
        &\E\left|
        \int_0^t|X_s^{(1),\ell_1^{\varepsilon_1}}(x)|^l\,
        \d W^{(2)}_{\ell_2^{\varepsilon_2}(s)
        -\ell_2^{\varepsilon_2}(0)}
        -\int_0^t|X_s^{(1),\ell_1}(x)|^l\,
        \d W^{(2)}_{\ell_2(s)}
        \right|^2\\
        &\quad\leq2\E\left|
        \int_0^t|X_s^{(1),\ell_1^{\varepsilon_1}}(x)|^l\,
        \d W^{(2)}_{\ell_2^{\varepsilon_2}(s)
        -\ell_2^{\varepsilon_2}(0)}
        -\int_0^t|X_s^{(1),\ell_1^{\varepsilon_1}}(x)|^l\,
        \d W^{(2)}_{\ell_2(s)}
        \right|^2\\
        &\quad\quad+2\E\left|
        \int_0^t|X_s^{(1),\ell_1^{\varepsilon_1}}(x)|^l\,
        \d W^{(2)}_{\ell_2(s)}
        -\int_0^t|X_s^{(1),\ell_1}(x)|^l\,
        \d W^{(2)}_{\ell_2(s)}
        \right|^2\\
        &\quad=2\E\left[
        \E_{\P(\cdot|\mathscr{F}^{(1)})}\left|
        \int_0^t|X_s^{(1),\ell_1^{\varepsilon_1}}(x)|^l\,
        \d W^{(2)}_{\ell_2^{\varepsilon_2}(s)
        -\ell_2^{\varepsilon_2}(0)}
        -\int_0^t|X_s^{(1),\ell_1^{\varepsilon_1}}(x)|^l\,
        \d W^{(2)}_{\ell_2(s)}
        \right|^2\right]\\
        &\quad\quad+2
        \int_0^t\E\left(|X_s^{(1),\ell_1^{\varepsilon_1}}(x)|^l
        -|X_s^{(1),\ell_1}(x)|^l\right)^2\,\d
        \ell_2(s)\\
        &\quad=:2I_1(\varepsilon_1,\varepsilon_2)
        +2I_2(\varepsilon_1).
    \end{aligned}
    \end{equation}
    According to \cite[Lemma 2.3\,(i)]{Zha13},
    \begin{equation}\label{hgf5f}
        \lim_{\varepsilon_2\downarrow0}
        \E_{\P(\cdot|\mathscr{F}^{(1)})}\left|
        \int_0^t|X_s^{(1),\ell_1^{\varepsilon_1}}(x)|^l\,
        \d W^{(2)}_{\ell_2^{\varepsilon_2}(s)
        -\ell_2^{\varepsilon_2}(0)}
        -\int_0^t|X_s^{(1),\ell_1^{\varepsilon_1}}(x)|^l\,
        \d W^{(2)}_{\ell_2(s)}
        \right|^2=0.
    \end{equation}
    Using \eqref{jh5dc4}, we find that
    \begin{align*}
        &\sup_{\varepsilon_2\in(0,1]}
        \E_{\P(\cdot|\mathscr{F}^{(1)})}\left|
        \int_0^t|X_s^{(1),\ell_1^{\varepsilon_1}}(x)|^l\,
        \d W^{(2)}_{\ell_2^{\varepsilon_2}(s)
        -\ell_2^{\varepsilon_2}(0)}
        -\int_0^t|X_s^{(1),\ell_1^{\varepsilon_1}}(x)|^l\,
        \d W^{(2)}_{\ell_2(s)}
        \right|^2\\
        &\leq2\sup_{\varepsilon_2\in(0,1]}
        \int_0^T|X_s^{(1),\ell_1^{\varepsilon_1}}(x)|^{2l}
        \,\d
        \ell_2^{\varepsilon_2}(s)
        +2\int_0^t|X_s^{(1),\ell_1^{\varepsilon_1}}(x)|^{2l}
        \,\d
        \ell_2(s)\\
        &\leq2^{1+(2l-1)^+}\sup_{\varepsilon_2\in(0,1]}
        \int_0^t\left(|x^{(1)}|^{2l}+
        \left|
        \int_0^s\sigma_r\,\d W^{(1)}
        _{\ell_1^{\varepsilon_1}(r)-\ell_1^{\varepsilon_1}(0)}
        \right|^{2l}\right)
        \,\d\ell_2^{\varepsilon_2}(s)\\
        &\quad+2^{1+(2l-1)^+}\int_0^t\left(|x^{(1)}|^{2l}+
        \left|
        \int_0^s\sigma_r\,\d W^{(1)}
        _{\ell_1^{\varepsilon_1}(r)-\ell_1^{\varepsilon_1}(0)}
        \right|^{2l}\right)
        \,\d\ell_2(s)\\
        &\leq2^{1+(2l-1)^+}
        \left(
        \sup_{\varepsilon_2\in(0,1]}\big[
        \ell_2^{\varepsilon_2}(t)-\ell_2^{\varepsilon_2}(0)
        \big]+
        \ell_2(t)
        \right)
        \left(|x^{(1)}|^{2l}+
        \sup_{s\in[0,t]}\left|
        \int_0^s\sigma_r\,\d W^{(1)}
        _{\ell_1^{\varepsilon_1}(r)-\ell_1^{\varepsilon_1}(0)}
        \right|^{2l}\right)\\
        &\leq2^{1+(2l-1)^+}
        \left(
        \ell_2^1(t)+
        \ell_2(t)
        \right)
        \left(|x^{(1)}|^{2l}+
        \sup_{s\in[0,t]}\left|
        \int_0^s\sigma_r\,\d W^{(1)}
        _{\ell_1^{\varepsilon_1}(r)-\ell_1^{\varepsilon_1}(0)}
        \right|^{2l}\right).
    \end{align*}
    This, together with Lemma \ref{BDG},
    yields
    $$
        \E\left[
        \sup_{\varepsilon_2\in(0,1]}
        \E_{\P(\cdot|\mathscr{F}^{(1)})}\left|
        \int_0^t|X_s^{(1),\ell_1^{\varepsilon_1}}(x)|^l\,
        \d W^{(2)}_{\ell_2^{\varepsilon_2}(s)
        -\ell_2^{\varepsilon_2}(0)}
        -\int_0^t|X_s^{(1),\ell_1^{\varepsilon_1}}(x)|^l\,
        \d W^{(2)}_{\ell_2(s)}
        \right|^2
        \right]<\infty.
    $$
    Then it follows from the dominated convergence
    theorem and \eqref{hgf5f} that
    \begin{equation}\label{limit1}
        \lim_{\varepsilon_2\downarrow0}
        I_1(\varepsilon_1,\varepsilon_2)=0.
    \end{equation}
    Next, we obtain
    from \eqref{jh5dc4} and Lemma \ref{BDG} that
    for any $q\geq0$
    \begin{align*}
         \sup_{\varepsilon_1\in(0,1],s\in[0,t]}
        \E|X_s^{(1),\ell_1^{\varepsilon_1}}(x)|^{q}
        &\leq2^{(q-1)^+}|x^{(1)}|^q+2^{(q-1)^+}
        \sup_{\varepsilon_1\in(0,1],s\in[0,t]}
        \E\left|\int_0^s\sigma_r\,\d W^{(1)}
        _{\ell_1^{\varepsilon_1}(r)-\ell_1^{\varepsilon_1}(0)}
        \right|^q\\
        &\leq2^{(q-1)^+}|x^{(1)}|^q+2^{(q-1)^+}
        \sup_{\varepsilon_1\in(0,1]}
        \left(
        \int_0^t\|\sigma_r\|^2\,\d\ell_1^{\varepsilon_1}(r)
        \right)^{q/2}\\
        &\leq2^{(q-1)^+}|x^{(1)}|^q+2^{(q-1)^+}
        \sup_{r\in[0,t]}\|\sigma_r\|^q\cdot
        \sup_{\varepsilon_1\in(0,1]}
        \left[
        \ell_1^{\varepsilon_1}(t)-
        \ell_1^{\varepsilon_1}(0)
        \right]^{q/2}\\
        &\leq2^{(q-1)^+}|x^{(1)}|^q+2^{(q-1)^+}
        \sup_{r\in[0,t]}\|\sigma_r\|^q\cdot
        \left[
        \ell_1^1(t)
        \right]^{q/2},
    \end{align*}
    and
    \begin{align*}
         \sup_{s\in[0,t]}
        \E|X_s^{(1),\ell_1}(x)|^{q}
        &\leq2^{(q-1)^+}|x^{(1)}|^q+2^{(q-1)^+}
        \sup_{s\in[0,t]}
        \E\left|\int_0^s\sigma_r\,\d W^{(1)}
        _{\ell_1(r)}
        \right|^q\\
        &\leq2^{(q-1)^+}|x^{(1)}|^q+2^{(q-1)^+}
        \|\sigma\|_{L^2([0,t];\,\d\ell_1)}^q.
    \end{align*}
    This means that for any $q\geq0$
    $$
        C_{x,q,t}:=\max\left\{
        \sup_{\varepsilon_1\in(0,1],s\in[0,t]}
        \E|X_s^{(1),\ell_1^{\varepsilon_1}}(x)|^{q},
        \,
        \sup_{s\in[0,t]}
        \E|X_s^{(1),\ell_1}(x)|^{q}
        \right\}
        <\infty.
    $$
    Using the elementary inequality
    $$
        \left|u^l-v^l\right|\leq\frac{l\vee2}{2}
        |u-v|^{l\wedge1}\left(u^{(l-1)^+}+
        v^{(l-1)^+}\right),\quad u,v\geq0,
    $$
    we obtain
    \begin{align*}
        &\frac{4}{(l\vee2)^2}\E\left(
        |X_s^{(1),\ell_1^{\varepsilon_1}}(x)|^l
        -|X_s^{(1),\ell_1}(x)|^l\right)^2\\
        &\leq\E\left[
        \big|X_s^{(1),\ell_1^{\varepsilon_1}}(x)-
        X_s^{(1),\ell_1}(x)\big|^{l\wedge1}
        \left(|X_s^{(1),\ell_1^{\varepsilon_1}}(x)|^{(l-1)^+}
        +|X_s^{(1),\ell_1}(x)|^{(l-1)^+}\right)
        \right]\\
        &\leq
        \sqrt{\E\big|X_s^{(1),\ell_1^{\varepsilon_1}}(x)-
        X_s^{(1),\ell_1}(x)\big|^{2(l\wedge1)}}
        \left(
        \sqrt{\E|X_s^{(1),\ell_1^{\varepsilon_1}}(x)|^{2(l-1)^+}}
        +\sqrt{\E|X_s^{(1),\ell_1}(x)|^{2(l-1)^+}}
        \right)\\
        &\leq2\sqrt{C_{x,2(l-1)^+,t}}
        \sqrt{\E\big|X_s^{(1),\ell_1^{\varepsilon_1}}(x)-
        X_s^{(1),\ell_1}(x)\big|^{2(l\wedge1)}},
    \end{align*}
    which, together with the dominated
    convergence theorem and \eqref{kj65fh}, yields
    $$
        \lim_{\varepsilon_1\downarrow0}
        I_2(\varepsilon_1)
        =\int_0^t\lim_{\varepsilon_1\downarrow0}
        \E\left(|X_s^{(1),\ell_1^{\varepsilon_1}}(x)|^l
        -|X_s^{(1),\ell_1}(x)|^l\right)^2\,\d
        \ell_2(s)=0.
     $$
     Combining this with \eqref{df3wa} and \eqref{limit1},
     we get \eqref{kj23sdc}.
\end{proof}

\begin{lemma}
    Assume \textbf{\upshape(A1)} and \textbf{\upshape(A2)}.
    Then for any $l>0$, $x\in\R^{m+d}$
    and $t>0$,
    \begin{equation}\label{fg34sd}
        \lim_{\varepsilon_1\downarrow0}
        \limsup_{\varepsilon_2\downarrow0}\E
        \big|X_t^{(2),\ell_1^{\varepsilon_1},
        \ell_2^{\varepsilon_2}}(x)
        -X_t^{(2),\ell_1,\ell_2}(x)\big|^2=0.
    \end{equation}
\end{lemma}

\begin{proof}
    Fix $l>0$, $x\in\R^{m+d}$ and $t>0$.
    It follows easily from \textbf{(A2)} that
    \begin{equation}\label{df34a}
    \begin{aligned}
        &\E\big|X_t^{(2),\ell_1^{\varepsilon_1},\ell_2^{\varepsilon_2}}(x)
        -X_t^{(2),\ell_1,\ell_2}(x)\big|^2
        \leq2tC_t^2\int_0^t
        \E\big|X_s^{(2),\ell_1^{\varepsilon_1},\ell_2^{\varepsilon_2}}(x)
        -X_s^{(2),\ell_1,\ell_2}(x)\big|^2\,\d s\\
        &\qquad\qquad\qquad+2\E\left|
        \int_0^t|X_s^{(1),\ell_1^{\varepsilon_1}}(x)|^l\,
        \d W^{(2)}_{\ell_2^{\varepsilon_2}(s)
        -\ell_2^{\varepsilon_2}(0)}
        -\int_0^t|X_s^{(1),\ell_1}(x)|^l\,
        \d W^{(2)}_{\ell_2(s)}
        \right|^2.
    \end{aligned}
    \end{equation}
    Since by \textbf{(A2)} $z\mapsto\sup_{s\in[0,t]}
    |b_s(z)|$ grows at most linearly, it is
    not hard to verify that
    $$
        \sup_{\varepsilon_1,\varepsilon_2\in(0,1],\,s\in[0,t]}
        \E|X_s^{(2),\ell_1^{\varepsilon_1},
        \ell_2^{\varepsilon_2}}(x)|^2<\infty,
        \quad
        \sup_{s\in[0,t]}
        \E|X_s^{(2),\ell_1,\ell_2}(x)|^2<\infty.
    $$
    Letting first $\varepsilon_2\downarrow0$
    and then $\varepsilon_1\downarrow0$ in \eqref{df34a},
    and using Fatou's lemma and \eqref{kj23sdc},
    we get
    \begin{align*}
        &\limsup_{\varepsilon_1\downarrow0}
        \limsup_{\varepsilon_2\downarrow0}
        \E\big|X_t^{(2),\ell_1^{\varepsilon_1},
        \ell_2^{\varepsilon_2}}(x)
        -X_t^{(2),\ell_1,\ell_2}(x)\big|^2\\
        &\qquad\qquad\qquad\qquad\quad\leq2tC_t^2\int_0^t
        \limsup_{\varepsilon_1\downarrow0}
        \limsup_{\varepsilon_2\downarrow0}
        \E\big|X_s^{(2),\ell_1^{\varepsilon_1},
        \ell_2^{\varepsilon_2}}(x)
        -X_s^{(2),\ell_1,\ell_2}(x)\big|^2\,\d s.
    \end{align*}
    This, together with Gronwall's inequality, yields
    the claim.
\end{proof}

\begin{lemma}\label{uniform}
    Assume \textbf{\upshape(A1)} and \textbf{\upshape(A2)}, and
    let $h:\R^{m+d}\rightarrow\R$ be a bounded
    and uniformly continuous function. Then
    for any $t>0$ and $x\in\R^{m+d}$
    $$
        \lim_{\varepsilon_1\downarrow0}
        \limsup_{\varepsilon_2\downarrow0}
        P_t^{\ell_1^{\varepsilon_1},
        \ell_2^{\varepsilon_2}}h(x)
        =P_t^{\ell_1,
        \ell_2}h(x).
    $$
\end{lemma}

\begin{proof}
    Fix $t>0$ and $x\in\R^{m+d}$. It follows from
    \eqref{kj65fh} and \eqref{fg34sd} that
    \begin{equation}\label{hg54cf}
        \lim_{\varepsilon_1\downarrow0}
        \limsup_{\varepsilon_2\downarrow0}
        \E\big|X_t^{\ell_1^{\varepsilon_1},
        \ell_2^{\varepsilon_2}}(x)-
        X_t^{\ell_1,
        \ell_2}(x)\big|^2=0.
    \end{equation}
    Since $h$ is uniformly continuous, for any $\epsilon>0$,
    there exists $\delta=\delta(\epsilon)>0$ such that
    $$
        |h(y)-h(z)|<\epsilon\quad\text{provided $|y-z|<\delta$.}
    $$
    Therefore, we obtain from Chebyshev's inequality that
    \begin{align*}
        &\big|P_t^{\ell_1^{\varepsilon_1},
        \ell_2^{\varepsilon_2}}h(x)-P_t^{\ell_1,
        \ell_2}h(x)\big|\\
        &\quad\qquad\qquad\leq\E\left[
        \big|h\big(X_t^{\ell_1^{\varepsilon_1},
        \ell_2^{\varepsilon_2}}(x)\big)
        -h\big(X_t^{\ell_1,
        \ell_2}(x)\big)\big|\I_{\big\{
        \big|X_t^{\ell_1^{\varepsilon_1},
        \ell_2^{\varepsilon_2}}(x)
        -X_t^{\ell_1,
        \ell_2}(x)\big|<\delta
        \big\}}
        \right]\\
        &\quad\qquad\qquad\quad+\E\left[
        \big|h\big(X_t^{\ell_1^{\varepsilon_1},
        \ell_2^{\varepsilon_2}}(x)\big)
        -h\big(X_t^{\ell_1,
        \ell_2}(x)\big)\big|\I_{\big\{
        \big|X_t^{\ell_1^{\varepsilon_1},
        \ell_2^{\varepsilon_2}}(x)
        -X_t^{\ell_1,
        \ell_2}(x)\big|\geq\delta
        \big\}}
        \right]\\
        &\quad\qquad\qquad<\epsilon
        +2\|h\|_\infty
        \P\left(\big|X_t^{\ell_1^{\varepsilon_1},
        \ell_2^{\varepsilon_2}}(x)
        -X_t^{\ell_1,
        \ell_2}(x)\big|\geq\delta\right)\\
        &\quad\qquad\qquad\leq\epsilon
        +2\|h\|_\infty\frac{\E\big|X_t^{\ell_1^{\varepsilon_1},
        \ell_2^{\varepsilon_2}}(x)
        -X_t^{\ell_1,
        \ell_2}(x)\big|^2}
        {\delta^2}.
    \end{align*}
    Letting first $\varepsilon_2\downarrow0$ and then
    $\varepsilon_1\downarrow0$, and
    using \eqref{hg54cf}, we get
    $$
        \limsup_{\varepsilon_1\downarrow0}
        \limsup_{\varepsilon_2\downarrow0}
        \big|P_t^{\ell_1^{\varepsilon_1},
        \ell_2^{\varepsilon_2}}h(x)-P_t^{\ell_1,
        \ell_2}h(x)\big|
        \leq\epsilon,
    $$
    which implies the claim since $\epsilon>0$
    is arbitrary.
\end{proof}

Now we are ready to prove Proposition \ref{change}.

\begin{proof}[Proof of Proposition \ref{change}]
    Fix $T>0$. By a standard
    approximation argument, we may and do
    assume that $f$ is bounded and uniformly continuous
    with $f\geq1$.

\emph{Step 1:}
    Assume \textbf{(A1)} and \textbf{(A2)}. Since
    $\ell_i$ is of bounded variation,
    it is not hard to verify from \eqref{app324} that
    $$
        \lim_{\varepsilon_1\downarrow0}\int_0^T\lambda_r^{-2}\,
        \d\ell^{\varepsilon_1}_1(r)
        =\int_0^T\lambda_r^{-2}\,
        \d\ell_1(r),
        \quad
        \lim_{\varepsilon_2\downarrow0}\int_T^{2T}\e^{-2K(T,s)}
        \,\d\ell^{\varepsilon_2}_2(s)
        =\int_T^{2T}\e^{-2K(T,s)}
        \,\d\ell_2(s),
    $$
    and
    $$
        \lim_{\varepsilon_2\downarrow0}\int_0^{T}\e^{-2K(0,s)}
        \,\d\ell^{\varepsilon_2}_2(s)
        =\int_0^{T}\e^{-2K(0,s)}
        \,\d\ell_2(s).
    $$
    Letting first $\varepsilon_2\downarrow0$ and then
    $\varepsilon_1\downarrow0$ in Lemma \ref{jhf5f},
    and using Lemma \ref{uniform},
    we get the desired log-Harnack inequality.

\emph{Step 2:}
    Assume \textbf{(A2)}. Clearly, we can pick a sequence
    of $\R^m\otimes\R^m$-valued
    functions $\{\sigma^{(n)}\,:\,n\in\N\}$ on $[0,\infty)$
    such that each $\sigma^{(n)}$ is piecewise constant,
    $\|(\sigma^{(n)}_t)^{-1}\|\leq\lambda_t$ for all $n\in\N$ and
    $t\in[0,2T]$, and $\sigma^{(n)}\rightarrow\sigma$ in
    $L^2([0,2T];\,\d\ell_1)$ as $n\rightarrow\infty$.
    Let $X_t^{\ell_1,\ell_2,n}(x)=(X_t^{(1),\ell_1,n}(x),
    X_t^{(2),\ell_1,\ell_2,n}(x))$ solve \eqref{changeeq}
    with $\sigma$ replaced by $\sigma^{(n)}$ and
    $X_0^{\ell_1,\ell_2,n}(x)=x\in\R^{m+d}$, and denote by
    $P_t^{\ell_1,\ell_2,n}$ the associated Markov semigroup.
    By Step 1, the statement of Proposition \ref{change}
    holds with $P_{2T}^{\ell_1,\ell_2}$
    replaced by $P_{2T}^{\ell_1,\ell_2,n}$. We have
    \begin{equation}\label{sdes23s}
    \begin{aligned}
        \lim_{n\rightarrow\infty}\E\big|X_t^{(1),\ell_1,n}(x)-
        X_t^{(1),\ell_1}(x)\big|^2
        &=\lim_{n\rightarrow\infty}
        \E\left|\int_0^{2T}\big(\sigma_t^{(n)}-\sigma_t\big)
        \,\d W^{(1)}_{\ell_1(t)}\right|^2\\
        &=\lim_{n\rightarrow\infty}
        \int_0^{2T}\big\|\sigma_t^{(n)}-\sigma_t\big\|
        ^2\,\d\ell_1(t)
        =0.
    \end{aligned}
    \end{equation}
    It holds from \textbf{(A2)} that
    \begin{align*}
        \big|X_{2T}^{(2),\ell_1,\ell_2,n}(x)
        -X_{2T}^{(2),\ell_1,\ell_2}(x)\big|
        &\leq C_{2T}\int_0^{2T}
        \big|X_t^{(2),\ell_1,\ell_2,n}(x)
        -X_t^{(2),\ell_1,\ell_2}(x)\big|\,
        \d t\\
        &\quad+\int_0^{2T}
        \big(|X_t^{(1),\ell_1,n}(x)|^l
        -|X_t^{(1),\ell_1}(x)|^l\big)\,\d t,
    \end{align*}
    which implies
    \begin{equation}\label{d3esc45r}
    \begin{aligned}
        \E\big|X_{2T}^{(2),\ell_1,\ell_2,n}(x)
        -X_{2T}^{(2),\ell_1,\ell_2}(x)\big|^2
        &\leq4TC_{2T}^2\int_0^{2T}
        \E\big|X_t^{(2),\ell_1,\ell_2,n}(x)
        -X_t^{(2),\ell_1,\ell_2}(x)\big|^2\,\d t\\
        &\quad+4T\int_0^{2T}\E\big(|X_t^{(1),\ell_1,n}(x)|^l
        -|X_t^{(1),\ell_1}(x)|^l\big)^2\,\d t.
    \end{aligned}
    \end{equation}
    Similarly as in the proof of Lemma \ref{kjh5fc3},
    we can deduce
    from \eqref{sdes23s} and Lemma \ref{BDG} that
    $$
        \lim_{n\rightarrow\infty}
        \int_0^{2T}\E\big(|X_t^{(1),\ell_1,n}(x)|^l
        -|X_t^{(1),\ell_1}(x)|^l\big)^2\,\d t
        =0.
    $$
    Letting $n\rightarrow\infty$ in \eqref{d3esc45r}
    and using Fatou's Lemma, we obtain
    \begin{align*}
        &\limsup_{n\rightarrow\infty}
        \E\big|X_{2T}^{(2),\ell_1,\ell_2,n}(x)
        -X_{2T}^{(2),\ell_1,\ell_2}(x)\big|^2\\
        &\qquad\qquad\qquad\qquad
        \leq4TC_{2T}^2\int_0^{2T}
        \limsup_{n\rightarrow\infty}
        \E\big|X_t^{(2),\ell_1,\ell_2,n}(x)
        -X_t^{(2),\ell_1,\ell_2}(x)\big|^2\,\d t,
    \end{align*}
    which, together with Gronwall's inequality, gives
    $$
        \lim_{n\rightarrow\infty}
        \E\big|X_{2T}^{(2),\ell_1,\ell_2,n}(x)
        -X_{2T}^{(2),\ell_1,\ell_2}(x)\big|^2
        =0.
    $$
    Then we conclude that for
    any $x\in\R^{m+d}$, $X_{2T}^{\ell_1,\ell_2,n}(x)
    \rightarrow X_{2T}^{\ell_1,\ell_2}(x)$ in $L^2(\P)$, and
    hence (up to a subsequence)
    $$
        \lim_{n\rightarrow\infty}
        P_{2T}^{\ell_1,\ell_2,n}f
        =P_{2T}^{\ell_1,\ell_2}f
        \quad \text{and}\quad
        \lim_{n\rightarrow\infty}
        P_{2T}^{\ell_1,\ell_2,n}\log f
        =P_{2T}^{\ell_1,\ell_2}\log f.
    $$
    Letting $n\rightarrow\infty$,
    the desired inequality in Proposition \ref{change} holds.

\emph{Step 3:}
    For the general case, we shall make use of
    the approximation argument
    in \cite[part (c) of proof of Theorem 2.1]{WW14}. Let
    $$
        \tilde{b}_t(z):=b_t(z)-k(t)z,\quad t\geq0,\,z\in\R^{d}.
    $$
    By \textbf{(H2)}, it is easy to see
    that the mapping
    $\operatorname{id}-n^{-1}\tilde{b}_t:\R^d\rightarrow\R^d$ is injective
    for any $n\in\N$ and $t\geq0$. Let
    $$
        b_t^{(n)}(z)=n\left[
        \left(\operatorname{id}
        -n^{-1}\tilde{b}_t\right)^{-1}(z)-z
        \right]+k(t)z,\quad n\in\N,\,t\geq0,\,z\in\R^d.
    $$
    Then we find that for any $n\in\N$ and $t\geq0$,
    $b_t^{(n)}$ is, uniformly for $t$ in compact
    intervals, globally Lipschitzian, see \cite{DRW09}. Let $\bar{X}_t^{\ell_1,\ell_2,n}(x)=(\bar{X}_t^{(1),\ell_1,n}(x),
    \bar{X}_t^{(2),\ell_1,\ell_2,n}(x))$ solve \eqref{changeeq}
    with $b$ replaced by $b^{(n)}$ and
    $\bar{X}_0^{\ell_1,\ell_2,n}(x)=x\in\R^{m+d}$, and denote by
    $\bar{P}_t^{\ell_1,\ell_2,n}$ the associated
    Markov semigroup.
    According to the second part of the
    proof, the statement of Proposition \ref{change}
    holds with $P_{2T}^{\ell_1,\ell_2}$
    replaced by $\bar{P}_{2T}^{\ell_1,\ell_2,n}$.
    As in \cite[part (c) of proof of Theorem 2.1]{WW14}, one has
    $\bar{X}_t^{\ell_1,\ell_2,n}(x)\rightarrow
    X_t^{\ell_1,\ell_2}(x)$ a.s. and therefore, it remains
    to let $n\rightarrow\infty$ to finish the proof.
\end{proof}

\section{Proofs of Theorem \ref{main1}
and Examples \ref{ex1} and \ref{ex2}}\label{sec3}

\begin{proof}[Proof of Theorem \ref{main1}]
    Noting that
    $$
        P_{2T}f(\cdot)=\E\left[P_{2T}^{\ell_1,\ell_2}f(\cdot)
        \left|_{\ell_1=S_1\atop \ell_2=S_2}\right.
        \right],\quad f\in\Bscr_b(\R^{m+d}),
    $$
    we get the desired log-Harnack inequality by
    using Proposition \ref{change} and
    the Jensen inequality.
\end{proof}

\begin{proof}[Proof of Example \ref{ex1}]
    By the self-similar property of $\alpha$-stable
    subordinators, one has
    \begin{equation}\label{moment1}
        \E S_1(T)^{-\kappa}=T^{-\kappa/\alpha}
        \E S_1(1)^{-\kappa},\quad T>0,\,\kappa>0.
    \end{equation}
    On the other hand, it is clear that
    \begin{equation}\label{moment2}
        \E S_2(T)=
        T\E S_2(1)<\infty,\quad T>0.
    \end{equation}
    Since $\frac{\e-1}{\e}(1\wedge z)\leq1-\e^{-z}
    \leq1\wedge z$ for all $z\geq0$, we get
    that for all $u>0$
    $$
        \int_{(0,1)}\left(1-\e^{-ux}\right)
        x^{-1-\beta}\,\d x\asymp u\int_0^1x^{-\beta}\,\d x
        =\frac{u}{1-\beta}.
    $$
    Here, $f\asymp g$ means that
    $c^{-1}f(u)\leq g(u)\leq c f(u)$ for some
    constant $c\geq 1$ and all $u$. This, together
    with \cite[Theorem 3.8\,(a)]{DS14} (or
    \cite[Theorem 2.1\,b)]{DSS16}), yields
    that for some constant $C_{\beta,c_2}>0$
    $$
        \E S_2(T)^{-1}\leq C_{\beta,c_2}T^{-1},\quad T>0.
    $$
    Combining the above estimates with Corollary
    \ref{cor}, we finish the proof.
\end{proof}

\begin{proof}[Proof of Example \ref{ex2}]
    Since $S_2$ has finite second moments,
    \eqref{moment2} holds true. Since the
    characteristic exponent of
    $S_2$ is $\phi_2(u)=c_2\beta^{-1}
    \Gamma(1-\beta)[(u+\rho^{1/\beta}
    )^\beta-\rho]$, we obtain
    from \cite[Theorem 2.1\,b)]{DSS16} that
    $$
        \E S_2(T)^{-1}\leq C_{\beta,c_2,\rho}\left(
        T^{-1/\beta}\vee T^{-1}\right)
        =C_{\beta,c_2,\rho} T^{-1}
        \left(
        T^{1-1/\beta}\vee1
        \right)
        ,\quad T>0
    $$
    for some constant $C_{\beta,c_2,\rho}>0$.
    Inserting this bound, \eqref{moment1} and
    \eqref{moment2} into Corollary \ref{cor}, the desired estimate follows.
\end{proof}

\section{Appendix}\label{app}

The following elementary result should
be known, but we could not
find a reference and so we include a simple
proof for the sake of completeness.

\begin{lemma}\label{ahf2}
    Let $\xi$ be an $m$-dimensional Gaussian
    random variable with mean zero and covariance
    matrix $\sigma I_{m\times m}$, where $\sigma>0$.
    Then for any $\theta\in(0,m/2)$, there exists
    $C=C(m,\theta)>0$ depending only
    on $m$ and $\theta$ such that
    $$
        \sup_{\mu\in\R^m}\E|\xi-\mu|^{-2\theta}
        \leq C\sigma^{-\theta}.
    $$
\end{lemma}

\begin{proof}
    First,
    \begin{align*}
        \E|\xi|^{-2\theta}&=\frac{1}{(2\pi\sigma)^{m/2}}
        \int_{\R^m}\frac{1}{|x|^{2\theta}}\,
        \e^{-|x|^2/(2\sigma)}\,\d x\\
        &=\frac{1}{\pi^{m/2}(2\sigma)^\theta}
        \int_{\R^m}\frac{1}{|y|^{2\theta}}\,
        \e^{-|y|^2}\,\d y\\
        &=c(m,\theta)\sigma^{-\theta},
    \end{align*}
    where $c(m,\theta):=\pi^{-m/2}2^{-\theta}\int_{\R^m}
    |y|^{-2\theta}\e^{-|y|^2}\,\d y$ is a positive
    constant since $2\theta<m$. For $\mu\in\R^m\setminus\{0\}$,
    \begin{equation}\label{hfd65d}
        \E|\xi-\mu|^{-2\theta}
        =I_1+I_2,
    \end{equation}
    where
    $$
        I_1:=\frac{1}{(2\pi\sigma)^{m/2}}
        \int_{|x-\mu|\geq|\mu|/2}\frac{1}{|x-\mu|^{2\theta}}\,
        \e^{-|x|^2/(2\sigma)}\,\d x,
    $$
    and
    $$
        I_2:=\frac{1}{(2\pi\sigma)^{m/2}}
        \int_{|x-\mu|<|\mu|/2}\frac{1}{|x-\mu|^{2\theta}}\,
        \e^{-|x|^2/(2\sigma)}\,\d x.
    $$
    If $|x-\mu|\geq|\mu|/2$, one has
    $$
        \frac{|x|}{|x-\mu|}\leq\frac{|x-\mu|+|\mu|}{|x-\mu|}
        \leq1+\frac{|\mu|}{|\mu|/2}=3,
    $$
    and then
    \begin{equation}\label{i1}
        I_1\leq\frac{1}{(2\pi\sigma)^{m/2}}
        \int_{|x-\mu|\geq|\mu|/2}
        \frac{3^{2\theta}}{|x|^{2\theta}}\,
        \e^{-|x|^2/(2\sigma)}\,\d x
        \leq3^{2\theta}\E|\xi|^{-2\theta}.
    \end{equation}
    On the other hand, if $|x-\mu|<|\mu|/2$, it
    follows that
    \begin{align*}
        |x|^2=|(x-\mu)+\mu|^2&=|x-\mu|^2
        +2\langle x-\mu,\mu\rangle+|\mu|^2\\
        &\geq|x-\mu|^2
        -2|x-\mu||\mu|+|\mu|^2\\
        &\geq|x-\mu|^2
        -2\frac{|\mu|}{2}|\mu|+|\mu|^2\\
        &=|x-\mu|^2.
    \end{align*}
    This implies
    \begin{equation}\label{i2}
        I_2\leq\frac{1}{(2\pi\sigma)^{m/2}}
        \int_{|x-\mu|<|\mu|/2}\frac{1}{|x-\mu|^{2\theta}}\,
        \e^{-|x-\mu|^2/(2\sigma)}\,\d x
        \leq\E|\xi|^{-2\theta}.
    \end{equation}
    Therefore, we obtain from \eqref{hfd65d}, \eqref{i1} and \eqref{i2} that, for any $\mu\in\R^m\setminus\{0\}$,
    $$
        \E|\xi-\mu|^{-2\theta}\leq\left(3^{2\theta}+1\right)
        \E|\xi|^{-2\theta}
        =\left(3^{2\theta}+1\right)
        c(m,\theta)\sigma^{-\theta},
    $$
    which completes the proof.
\end{proof}

\begin{ack}
    The authors would like to thank an
    anonymous referee for useful suggestions.
\end{ack}

\end{document}